\documentclass[11pt, reqno]{article}

\usepackage{fullpage}
\usepackage{enumerate}
\usepackage{setspace}

\usepackage{upgreek, mathtools,bm}

\usepackage{amsmath,amsfonts,amssymb,graphics,amsthm,comment}
\usepackage{hyperref}
\hypersetup{
    colorlinks=true,
    linkcolor=red,
    citecolor=blue,
    urlcolor=blue,
    pdfborder={0 0 0}
}
\usepackage{yfonts}

\usepackage{graphicx}
\usepackage{xcolor}
\usepackage{bbm}
\usepackage{wrapfig} 

\usepackage[font=sf, labelfont={sf,bf}, margin=1cm]{caption}

\usepackage{cleveref}
  \crefname{theorem}{Theorem}{Theorems}
  \crefname{conjecture}{Conjecture}{Conjectures}
  \crefname{thm}{Theorem}{Theorems}
  \crefname{thm*}{Theorem*}{Theorems}
  \crefname{lemma}{Lemma}{Lemmas}
  \crefname{lem}{Lemma}{Lemmas}
  \crefname{remark}{Remark}{Remarks}
  \crefname{prop}{Proposition}{Propositions}
\crefname{notation}{Notation}{Notations}
\crefname{claim}{Claim}{Claims}
  \crefname{defn}{Definition}{Definitions}
  \crefname{corollary}{Corollary}{Corollaries}
  \crefname{section}{Section}{Sections}
  \crefname{figure}{Figure}{Figures}
    \crefname{assumption}{Assumption}{Assumptions}

\newtheorem{thm}{Theorem}[section]
\newtheorem{thm*}{Theorem*}[section]

\newtheorem{conjecture}[thm]{Conjecture}
\newtheorem{claim}[thm]{Claim}
\newtheorem{lemma}[thm]{Lemma}
\newtheorem{corollary}[thm]{Corollary}
\newtheorem{prop}[thm]{Proposition}

\newtheorem{question}[thm]{Question}

\numberwithin{equation}{section}

\theoremstyle{definition}

\def\cS{\mathcal{S}}
\def\cR{\mathcal{R}}

\def\cO{\mathcal{O}}
\def\cN{\mathcal{N}}

\def\cL{\mathcal{L}}

\def\cI{\mathcal{I}}

\def\cE{\mathcal{E}}

\def\cB{\mathcal{B}}
\def\cA{\mathcal{A}}
\def \ve {\varepsilon}

\def\P{\mathbb{P}}

\def\R{\mathbb{R}}
\def\Z{\mathbb{Z}}
\def\N{\mathbb{N}}

\def\T{\mathbb{T}}

\def  \p- {p\textunderscore}

\def\1{\mathbf{1}}

\def\eps{\varepsilon}

\def\zero{\bm{0}}

\title{On the local convergence of integer-valued Lipschitz functions on regular trees.}
\author{Nathaniel Butler and Kesav Krishnan and Gourab Ray and Yinon Spinka}

\date{\today}
\author{Nathaniel Butler\thanks{University of Victoria. Research supported in part by NSERC 50311-57400.} \and Kesav Krishnan\footnotemark[1]\thanks{University of Victoria. Research supported in part by PIMS postdoctoral fellowship} \and Gourab Ray \footnotemark[1]\thanks{ University of Victoria. Research supported in part by NSERC 50311-57400. Email:gourabray@uvic.ca } \and Yinon Spinka\thanks{Tel Aviv University. Research supported in part by ISF grant 1361/22} }
\begin{document}
\maketitle
\begin{abstract}
We study random integer-valued Lipschitz functions on regular trees. It was shown by Peled, Samotij and Yehudayoff~\cite{PWA_tree} that such functions are localized, however, finer questions about the structure of Gibbs measures remain unanswered.
Our main result is that the weak limit of a uniformly chosen 1-Lipschitz function with 0 boundary condition on a $d$-ary tree of height $n$ exists as $n \to \infty$ if $2 \le d \le 7$, but not if $d \ge 8$, thereby partially answering a question from~\cite{PWA_tree}. 
For large $d$, the value at the root alternates between being almost entirely concentrated on 0 for even $n$ and being roughly uniform on $\{-1,0,1\}$ for odd $n$, leading to different limits as $n$ approaches infinity along evens or odds. For $d \ge 8$, the essence of this phenomenon is preserved, which obstructs the convergence. For $d \le 7$, this phenomenon ceases to exist, and the law of the value at the root loses its connection with the parity of $n$. Along the way, we also obtain an alternative proof of localization.
The key idea is a fixed point convergence result for a related operator on $\ell^\infty$, and a procedure to show that the iterations get into a `basin of attraction' of the fixed point.

We also prove some accompanying analogous `even-odd phenomenon' type results about $M$-lipschitz functions on general non-amenable graphs with high enough expansion (this includes for example the large $d$ case for regular trees). We also prove a convergence result for 1-Lipschitz functions with $\{0,1\}$ boundary condition. This last result relies on an absolute value FKG for uniform 1-Lipschitz functions when shifted by $1/2$.
\end{abstract}

\section{Introduction}
The main focus of this article is to study the structure of Gibbs measures of uniform $M$-Lipschitz function on regular trees. Our broader goal is to better understand the structure of Gibbs measures  for random surface models with hardcore constraints on graphs of hyperbolic flavour. The theory of Bernoulli percolation on such graphs  has seen tremendous progress in the past few decades \cite{lyons2017probability,bperc96}. In contrast, the results on  other popular statistical physics models like the Ising model on such graphs have been few and far between \cite{h_ising,HSS_00,JS_99,georgii2011gibbs}. On the other hand, the study of height function models on planar lattices, particularly the structure of Gibbs measures on them,  have garnered a lot of attention in recent years \cite{chandgotia2018delocalization,duminil2022logarithmic,LT_24,peled2017high}. For finite graphs like expanders and wired regular trees, it is known \cite{PWY_expander,PWA_tree,BHM_00,G_03} that the height function models are mostly flat and the height function has a doubly exponential tail. However, finer questions about the existence of local weak limits and its dependence on boundary conditions remain unanswered. Our main focus is to initiate the study of weak local convergence of height functions on regular trees imposing various boundary conditions on the leaves. An even broader goal is to classify all possible extremal Gibbs measures on trees. 

Let us fix some notations before getting into  our main results. An \textbf{$M$-Lipschitz function} on a graph $G=(V,E)$ is a function $f \colon V \to \Z$ satisfying that $$|f(u) - f(v)| \le M\text{ for all }u \sim v,$$ where $u \sim v$ means $u$ is a neighbour of $v$. If $M=1$, we drop the $M$ from the terminology, and simply refer to the function as a Lipschitz function. A $d$-ary tree is a connected graph with no cycles such that every vertex has degree $d+1$ except one vertex, called the root vertex, which has degree~$d$. We denote by $\mathbb T_d$ the $d$-ary tree\footnote{Most of our stated results continue to hold with no changes also for the $(d+1)$-regular tree, see \Cref{sec:open}.} with the root vertex denoted $\rho$. Let $\T_d^n$ denote the finite subgraph of $\T_d$ induced by the vertices within distance $n$ from $\rho$. For a finite set $S \subset \Z$, let $\cL_{n,d}(S)$ denote the set of Lipschitz functions on $\T_d^n$ where the values on the leaves are constrained to be in $S$. Let $\mu_{n,d}^{S}$ denote the uniform law on $\cL_{n,d}(S)$. In the special case $S= \{0\}$, which we refer to as zero boundary conditions, we write $\mu_{n,d}^0$ and $\cL_{n,d}^0$ for short.


The main result in this article is the existence of a transition in the convergence behavior of random 1-Lipschitz functions under zero boundary conditions.
\begin{thm}\label{thm:main_lipschitz}
    $\mu^0_{n,d}$ converges as $n \to \infty$ if and only if $2 \le d \le 7$. 
\end{thm}

The convergence above is in the local weak sense.
Actually, the convergence result for $2 \le d \le 7$ in \Cref{thm:main_lipschitz} that we prove is more general: it holds for symmetric weighted boundary conditions on the leaves which satisfy a certain decay condition (we call these the good weight sequences in \Cref{sec:convergence}).
The result implies in particular  that for any $k \ge 1$, the measures $\mu^{\{-k,\dots,k\}}_{n,d}$ converge to the same limit as that with zero boundary conditions. See \Cref{thm:main_conv2} for a complete statement of the full result.

In contrast, the next result shows that for other boundary conditions, the above convergence vs. non-convergence phenomenon does not occur. 
\begin{thm}\label{thm:main_lipschitz_FKG}
$\mu^{\{0,1\}}_{n,d}$ converges as $n \to \infty$ for all $d \ge 2$.
\end{thm}
The above theorem is a consequence of a version of an FKG property for the absolute value of the Lipschitz function \emph{when shifted by $1/2$}.

It is easy to see for any finite $S$  that $\lim_{n \to \infty}\mu^S_{n,d}$, when it exists, is a tree-indexed Markov chain. Furthermore, it is not hard to check that no such Markov chain (obtained as a limit of $\mu^S_{n,d}$) is a non-trivial mixture of other similar Markov chains (simply apply ergodic theorem along a ray). In particular, when $2 \le d \le 7$, the measures $\{\mu^{i}_{d}\}_{i \in \Z}$ obtained in \cref{thm:main_lipschitz} and the measures $\{\mu^{\{i,i+1\}}_d\}_{i \in \Z}$ obtained in \cref{thm:main_lipschitz_FKG} are Markov chains, none of which is a mixture of the others.

\medskip

We also prove a general result which works beyond trees and for general $M$. Let $G$ be an infinite bipartite connected graph with maximum degree $d$ and Cheeger constant $h$. Recall that the Cheeger constant of a graph is defined as 
$$
h=h(G):=\inf_{S \subset V, 0<|S|<\infty} \frac{|\partial S|}{|S|},
$$
where $\partial S$ is the set of vertices in $G$ which are not in $S$ but have at least one neighbour in $S$, and $|\cdot|$ denote the cardinality. Color the partite classes of $G$ black and white. We establish the existence of (multiple forms of) long-range order for $M$-Lipschitz functions on $G$ in the regime where $h \gg M,d$. Observe that since trivially $h \le d$, the condition $h \gg M$ requires that $d \gg M$. Furthermore, note that a $d$-regular tree has Cheeger constant $d-2$.


    Given a finite set $U$ of $G$, and finite sets $S_\circ,S_\bullet \subset \Z$, let $\cL_{U,M}(S_\circ,S_\bullet)$ denote the set of all $M$-lipschitz functions on $U \cup \partial U$ with the values on $\partial U$ constrained to be in $S_\circ$ or $S_\bullet$ according to their color. Given integers $a \le b$, let $\nu^{a,b}_{U,M}$ denote the uniform measure on $\cL_{U,M}(S_\circ,S_\bullet)$ with $S_\circ=\{a,\dots,b\},S_\bullet=\{b-M,\dots,a+M\}$.
    An exhaustion of $G$ is a sequence of finite sets increasing to $G$, i.e., $G_1\subset G_2\subset \cdots$ and $\cup_{i \ge 1} G_i = G$.

\begin{thm}\label{thm:main_general}
   Assume that $h \ge 4M\log(3d^4(4M+1))$.
       Let $a \le b$ be integers such that $b-a \le 2M$.
       \begin{itemize}
           \item [a.] Let $(G_n)$ be an exhaustion. Then $\nu_{G_n,M}^{a,b}$ converges as $n\to \infty$ to a Gibbs measure $\nu_{G,M}^{a,b}$.
           \item [b.] The measures $\nu_{G,M}^{a,b}$ are distinct for different pairs $(a,b)$.
       \end{itemize}
   \end{thm}

    The limiting measure $\nu_{G,M}^{a,b}$ is necessarily independent of the chosen exhaustion (as one can interleave two exhaustions and then apply the result). Observe that the pushforward of $\nu_{U,M}^{a,b}$ by a parity-preserving automorphism $\varphi$ of $G$ is $\nu_{\varphi(U),M}^{a,b}$. It follows that the limiting measure $\nu_{G,M}^{a,b}$ is invariant to parity-preserving automorphisms of $G$.
   Similarly, the pushforward of $\nu_{G,M}^{a,b}$ by a parity-reversing automorphism is $\nu_{G,M}^{b-M,a+M}$. Thus, $\nu_{G,M}^{a,b}$ is invariant to all automorphisms if and only if $b-a=M$. Consequently, when $b-a \neq M$, the measure $\nu_{G,M}^{a,b}$ behaves differently on even and odd vertices. In fact, the height on any even vertex is close to uniform on $\{a,\dots,b\}$ and the height on any odd vertex is close to uniform on $\{b-M,\dots,a+M\}$ (see \cref{prop:cluster_exp_tail}). Thus, \cref{thm:main_general} gives a separate proof of the non-convergence result in \cref{thm:main_lipschitz} when $d$ is large enough.
   Let us also remark that in the case when $b-a=M$, the result stated in the theorem makes sense and holds true even without the assumption that the graph $G$ is bipartite (in this case $\nu^{a,b}_{U,M}$ is the uniform measure on $M$-Lipschitz functions constrained to taking values in $\{a,\dots,b\}$ on the entire boundary).

\paragraph{Background:}
Studying Gibbs measures of statistical physics models on trees has a long history. We refer to the book of Rozikov \cite{rozikov_book} for a comprehensive survey of the results about splitting Gibbs measures on regular trees for models with countable and even uncountable spin values. It is interesting to note
that the results in \cite{rozikov_book} about splitting Gibbs measure for countable spins (see \cite[Theorem 8.1]{rozikov_book}) lead to a unique Gibbs measure. This is clearly not the case for us, as \Cref{thm:main_lipschitz} (see also \cref{thm:main_conv2}) already leads to multiple solutions, and it is plausible that there are many more (see \Cref{sec:open}). 

The random homomorphism model is very similar to the uniform Lipschitz model, except that the height difference between adjacent vertices must be in $\{\pm 1\}$ (this makes sense only on bipartite graphs).
Both the uniform Lipschitz model and the random homomorphism model on trees have been studied in \cite{PWA_tree,BHM_00} and on finite expanders in \cite{PWY_expander}. A recent paper by Lammers and Toninelli~\cite{LT_24} exploits log concavity to obtain localization for homomorphisms with very general boundary conditions. 
See also a recent result by Krueger, Li and Park \cite{KLP_24} which significantly improves the  results in \cite{PWY_expander}. See \cite{KKR_widom} for a result about the Widom--Rowlinson model on trees, which is a model of similar flavour to ours (it can be thought of as a restricted Lipschitz model where the function cannot take values other than $\{-1,0,1\}$). See Peled and Spinka \cite{peled2017high,PS_rigidity} for analogous results in the hypercubic lattice.
See also \cite{BS_00,BYY_07,BP_tree_index} for related results.

As mentioned, the fact that random Lipschitz functions under zero boundary conditions are localized was first proved by Peled, Samotij and Yehudayoff~\cite{PWA_tree}. In fact, they showed that the height at the root has doubly exponential tails for all $M \ge 1$ and $d \ge 2$.
They also point out\footnote{This was announced in~\cite{PWA_tree}, but was never published. \cref{thm:main_general} is our attempt to write down a precisely formulated general result and give the details of the proof.} that when $M \le cd/\log d$ for some universal constant $c>0$, more is true --- for even $n$, the height at the root is exponentially concentrated on the single value $0$, whereas for odd $n$, the height is at the root is roughly uniformly distributed on $\{-M,\dots,M\}$. They go on to explain that they expect that such a strong concentration fails when $M\gg d$. Specifically, they raised their suspicion that the height at the root converges in distribution as $n$ tends to infinity (so that there is asymptotically no distinction between even and odd tree heights). While we do not establish that such a transition phenomenon occurs between the cases $M\ll d$ and $M \gg d$ for all $M$, we do show in \cref{thm:main_lipschitz} that such a transition occurs as $d$ is varied when $M=1$. Extending this to all $M$ is an interesting problem.

\paragraph{Proof outline:}
Let us briefly outline the ideas behind the proofs. \Cref{thm:main_lipschitz} is the major result in this article and takes up most of the analysis. There are two separate parts, convergence for $2 \le d \le 7$ and non-convergence for $d \ge 8$, the latter being much simpler. The basic idea for the convergence part is a contraction argument: we think of the probability distribution at the root as an element of $\ell^1(\Z)$ and then we apply iterates of an appropriate operator $F$. The goal then reduces to proving that the iterates of $F$ (applied to $\delta_0$) converge in an appropriate norm. The proof of this would have been somewhat straightforward if the map $F$ turned out to be a contraction (but it isn't). Nevertheless, we show that after a number of iterations, we get into a certain ``basin of attraction'', on which the map becomes a contraction. To show the latter, we prove that the operator norm of the derivative of $F$ is strictly less than 1 in this basin. After this, the proof concludes by a simple application of the Banach fixed point theorem.

In practice, directly applying this idea to $F$ soon becomes intractable. The way we work our way through is to look at ratios of the consecutive values, which gets rid of the normalizing factor in a good way. For the non-convergence part of the theorem, we show that the iterates oscillate between two separated `basins', ruling out the possibility of convergence (though we still expect convergence of $\mu^0_{n,d}$ along even $n$). The details of these ideas can be found in \Cref{sec:conv_non_conv}.

\Cref{thm:main_lipschitz_FKG} follows from a certain ``FKG for absolute value'' result when the height function is shifted by $1/2$. Finally \Cref{thm:main_general} is a consequence of a variation of a Peierl's type argument following that in \cite{PWA_tree,PWY_expander}.

\paragraph{Organization:} In \Cref{sec:conv_non_conv}, we describe the setup and the ratio transformation which simplifies the operator $F$. We also describe heuristically which observables from this operator drives the transition from convergence to non-convergence.  In \Cref{sec:envelope} we define the set which is the `basin of attraction' and describe our procedure to get into this set. In \Cref{sec:nonconvergence} we prove the non convergence part of \Cref{thm:main_lipschitz}. In \Cref{sec:convergence,sec:apriori_bounds,sec:contraction} we prove the convergence part of \Cref{thm:main_lipschitz} modulo some technical estimates, which we push to \Cref{sec:phi_estimate}. We prove \Cref{thm:main_lipschitz_FKG} in \Cref{sec:FKG} and \Cref{thm:main_general} in \Cref{sec:long_range}. We list some open questions in \Cref{sec:open}.
\paragraph{Acknowledgements:} We thank Omer Angel, Jinyoung Park and Ron Peled for several illuminating discussions. KK would like to thank the Salmon Coast Field Station for hosting him while a portion of this work was completed. 


\bigskip

 

\section{Convergence and non-convergence under 0 boundary conditions}\label{sec:conv_non_conv}

In this section we prove \cref{thm:main_lipschitz}. The proof is separated into the convergence result for $2 \le d \le 7$ (\cref{sec:convergence}) and the non-convergence result for $d \ge 8$ (\cref{sec:nonconvergence}).
The proofs are based on a recursive approach. We now lay the groundwork for this.

For a probability distribution $z=(z_i)_{i \in \Z}$, define
\[ A(z)_i := \left(z_{i-1}+z_i+z_{i+1}\right)^d , \qquad F(z) := \frac{A(z)}{\sum_{i \in \Z} A(z)_i}. \]
The following claim is immediate from the definition.
\begin{claim}\label{lem:F}
Let $f_n \sim \mu_{n,d}^0$ and let $\rho$ be the root vertex. Then for any $k \in \Z$,
$$
\P(f_n(\rho) = k) = (F^{(n)}(\delta_0))_{k} .
$$
\end{claim}

Thus, vague convergence of the distribution of $f_n(\rho)$ is determined by the pointwise convergence of $F^{(n)}(\delta_0)$ (tightness is a separate issue which needs to be addressed). In order to understand the latter, our strategy is to perform a change of coordinates which gets rid of the normalizing factor in~$F$, and thereby makes things more amenable to analysis. Define a map ${\mathsf R} \colon [0,\infty)^\N \to [0,\infty)^\N$ by
\begin{equation}\label{eq:phi}
    {\mathsf R}(z)_i = 
    \begin{cases}
        \frac{z_i}{z_{i-1}} &\text{if } z_{i-1} \neq 0\\
        0 &\text{if }z_{i-1}=0.
    \end{cases}
\end{equation}
In order for ${\mathsf R}$ to be injective, we restrict its domain to the set $\cE$ of symmetric probability distributions whose support is an interval or all of $\Z$, i.e., probability distributions $z$ such that $z_i=z_{-i}$ for all $i>0$, and $z_i=0$ implies that $z_{i+1}=1$ for $i>0$. Let $\cR$ be the image of $\cE$ under ${\mathsf R}$. Observe that $\varphi$ is indeed injective and hence invertible.
Define a new map $\psi:\cR \to \cR$ by 
$$
\psi := {\mathsf R} \circ F \circ {\mathsf R}^{-1}.
$$
Observe that $F$ preserves $\cE$, i.e., $\varphi(z) \in \cE$ if $z \in \cE$.
Hence, $\psi$ preserves $\cR$.
Thus,
\[ \psi^{(n)} = {\mathsf R} \circ F^{(n)} \circ {\mathsf R}^{-1} \qquad\text{and}\qquad F^{(n)} = {\mathsf R}^{-1} \circ \psi^{(n)} \circ {\mathsf R} .\]
Noting that ${\mathsf R}(\delta_0)=\zero:=(0,0,\dots)$, and that ${\mathsf R}^{-1}$ is continuous (with respect to the pointwise topologies), we see that in order to establish the pointwise convergence of $F^{(n)}(\delta_0)$, it is sufficient to show the pointwise convergence of $\psi^{(n)}(\zero)$. The converse is also true so that pointwise convergence of these sequences is equivalent, since ${\mathsf R}$ is continuous on the subset of $\cE$ of fully supported distributions (where the limit must clearly reside), but we will not use this.

Let us record already here the expressions for the coordinates of $\psi$.
Writing $\psi_n(\cdot)=\psi(\cdot)_n$, we have by writing ${\mathsf R}^{-1}(x) = z$,
\begin{equation}\label{eq:psi_1}
\psi_1(x) = \frac{(z_0+z_1+z_2)^d}{(z_{-1}+z_0+z_1)^d} = \left(\frac{1+x_1+x_1x_2}{1+2x_1}\right)^d ,
\end{equation}
where we used the fact that $z_1=z_{-1}$ since $z \in \cE$.
Similarly, for $n \ge 2$, we get
\begin{equation}\label{eq:psi_n}
\psi_n(x) = x_{n-1}^d \left(\frac{1+x_n+x_nx_{n+1}}{1+x_{n-1}+x_{n-1}x_n}\right)^d .
\end{equation}

Let us now provide a rough explanation of what properties of these functions change from the convergence phase ($d \le 7$)  to the non-convergence phase ($d \ge 8$). 
For $1\le \alpha <  2$ and $0 \le x \le 1$, define  
\begin{align}
    f(\alpha,x) &:= \left(\frac{1+\alpha x}{1+2x}\right)^d. \label{eq:f}
\end{align}
Letting $(1+x_2)= \alpha$, we see that $\psi_1$ is the same as the map $f(\alpha, x)$.
It is known \cite{PWA_tree} that the uniform Lipschitz function at the root has double exponential decay, which means that $x_2$ is small. Thus it is reasonable to look at $f(1,x)$ as a test case. Although $f(1,x)$ is decreasing and convex, and has a single fixed point $x_*$, $|f'(1, x_*)|>1$ if and only if $d \ge 8$. Furthermore, it turns out that $f \circ f$ starts having multiple fixed point for $d \ge 8$. Also, $|(f \circ f)'|<1$ for all $0 \le x \le 1$ if $2 \le d \le 7$, for $d \ge 8$, $(f \circ f)'(x_*) >1$. These facts can be proved directly, we show a few plots here in \Cref{Plots}.
\begin{figure}[h]
    \centering
    \includegraphics[width=0.25\linewidth]{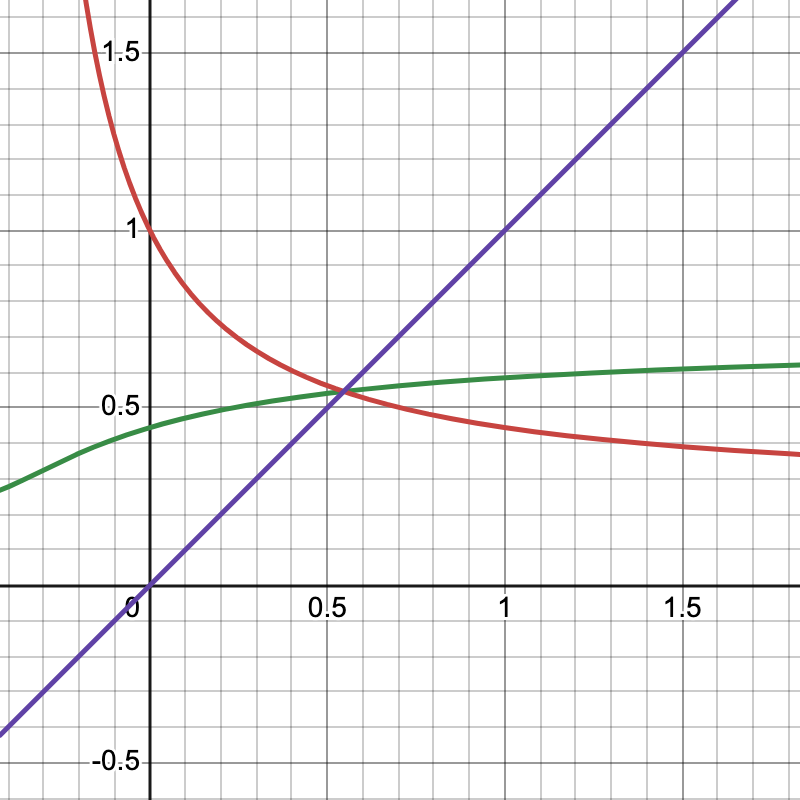}\qquad\qquad
        \includegraphics[width=0.25\linewidth]{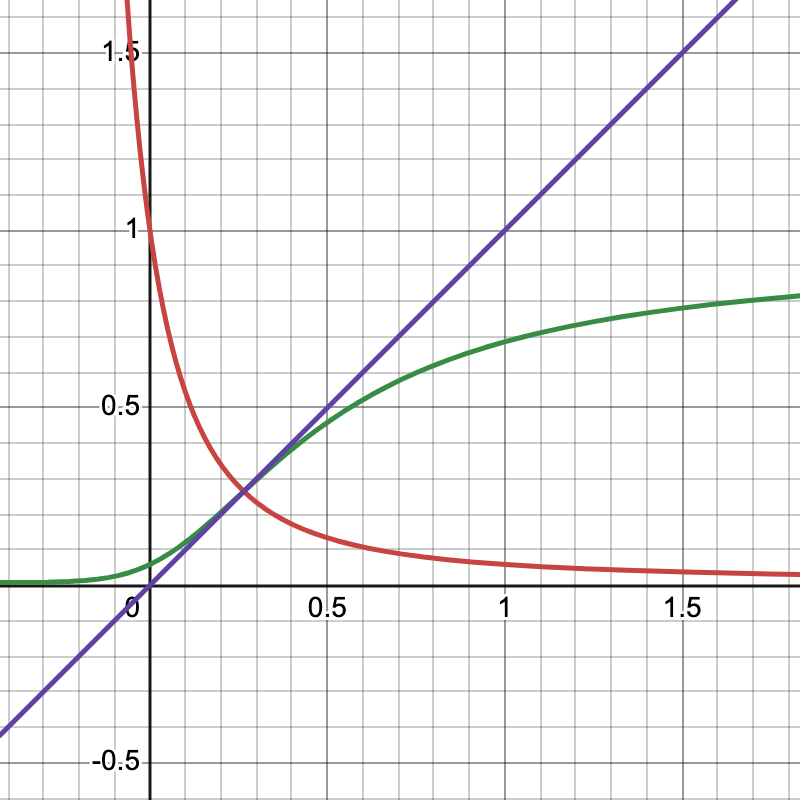}\\
        \vspace{1 cm}
        \includegraphics[width=0.25\linewidth]{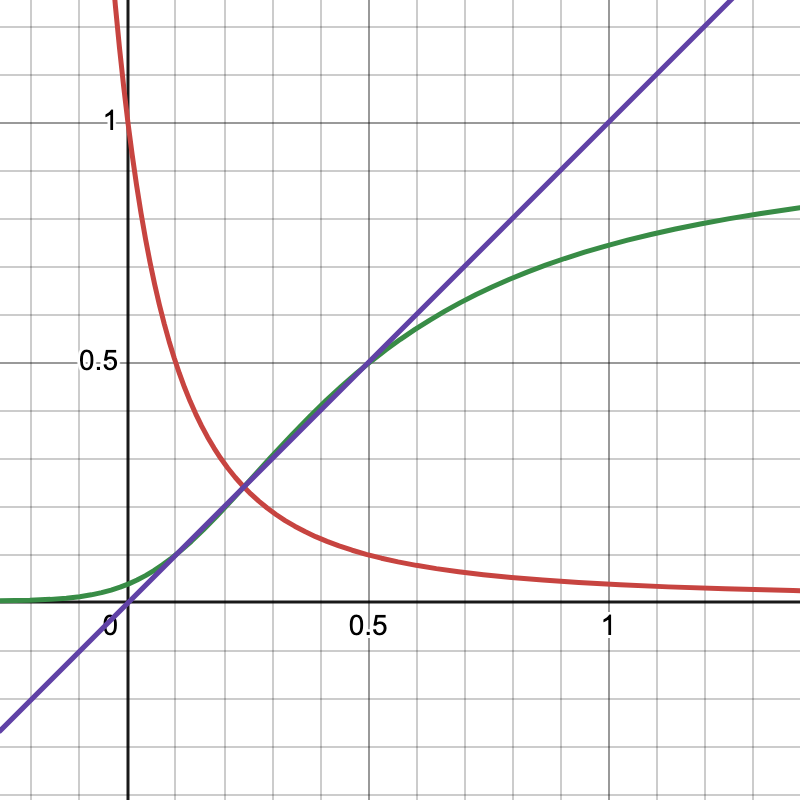}\qquad \qquad
        \includegraphics[width=0.25\linewidth]{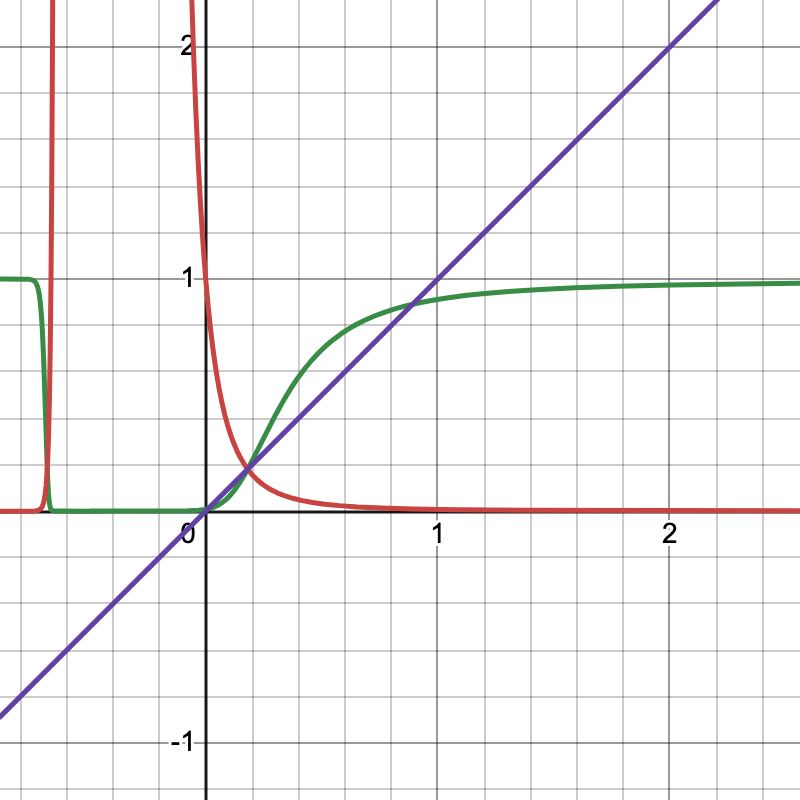}
    \caption{Plots of $f=f(1,x)$ (in red), $f \circ f$ (in green) and $y=x$ (in blue) for different values of $d$, top left: $d=2$, top right $d=7$, bottom left: $d=8$, bottom right $d=12$. Multiple fixed points of $f \circ f $ start appearing for $d \ge 8$.}
    \label{Plots}
\end{figure}
In reality however, $x_2$ is not constant but fluctuates. Therefore, to prove convergence, we need to look at the full operator norm.

The rest of this section is organized as follows.
We first construct an ``envelope dynamics'' in \cref{sec:envelope} which gives a decent approximation of the full dynamics of $\psi$. We then use this to prove non-convergence for $d \ge 8$ in \cref{sec:nonconvergence}. Finally, we continue on to the longer and more intricate proof of convergence for $2 \le d \le 7$ in \cref{sec:convergence}, for which the envelope dynamics only serves as an initial step to obtain an approximation of the fixed point.

\subsection{Envelope dynamics}\label{sec:envelope}

We now describe subsets of $\cR$ which capture much of the essence of the dynamics given by $\psi$. Because of the fast decay, we manage to describe these sets via three parameters, which give an interval for the range of the first coordinate and a maximum for the rest of the coordinates.
For numbers $0 \le a,c \le b \le 1$, denote
\begin{align*}
 \cS_{a,b,c} &:= \Big\{ x \in \cR : a \le x_1 \le b,~ 0 \le x_n \le c \text{ for all } n \ge 2 \Big\} .
\end{align*}
Define $\varphi \colon [0,1]^3 \to [0,1]^3$ by
\begin{equation}
    \varphi(a,b,c) := \left( \left(\frac{1+b}{1+2b}\right)^d, \left(\frac{1+a+ac}{1+2a}\right)^d, b^d \left(\frac{1+c+c^2}{1+b+bc}\right)^d \right).
\end{equation}
The following lemma describes the evolution of the sets $\cS_{a,b,c}$ under applications of $\psi$.
Denote
\[ \cI := \Big\{ (a,b,c) \in [0,1]^3 : a,c \le b \Big\} .\]

\begin{lemma}\label{lem:S-abc-iteration}
    Let $(a,b,c) \in \cI$. Then $(a',b',c'):=\varphi(a,b,c) \in \cI$ and $\psi(\cS_{a,b,c}) \subset \cS_{a',b',c'}$.
\end{lemma}

Before proving \Cref{lem:S-abc-iteration}, it is useful to make some observations. For $0 \le b,x \le 1$, define
\begin{align}
    g(b,x) &:= b^d \left(\frac{1+x+x^2}{1+b+bx}\right)^d. \label{eq:g}
\end{align}
Note that with these definitions we have
\[ \varphi(a,b,c) := (f(1,b), f(1+c,a), g(b,c)) .\]
The following are simple observations. For $x, \alpha,\beta,b>0$ 
\begin{align}
    &\frac{1+\alpha x}{1+\beta x} \text{ is decreasing (resp. increasing) in $x$ if $\alpha<\beta$ (resp. $\alpha>\beta$)}.\label{eq:f'}\\
& g(b,x) \text{ is increasing in $b$}.\label{eq:g'}
\end{align}
To verify \eqref{eq:g'} simply  divide the the numerator and the denominator by $b^d$ and observe that after cancellations, the resulting expression is increasing in $b$.
Before proving the main part of \cref{lem:S-abc-iteration}, let us show that $\varphi$ preserves $\cI$.
\begin{claim}
    $\varphi$ preserves $\cI$, i.e., $f(\cI) \subset \cI$.
\end{claim}
\begin{proof}
    Let $(a,b,c) \in \cI$ and set $(a',b',c') := \varphi(a,b,c)$.
    We need to check that $a' \le b'$ and $c' \le b'$.
    For the first, we note $f(1,b) \le f(1,a) $ using \eqref{eq:f'} and clearly $f(1,a) \le f(1+c,a)$
    For the second, we need to check that $g(b,c) \le f(1+c,a)$.
    First note that $g(b,c) \le g(1,c)$ by \eqref{eq:g'}. Next, to see that $g(1,c) \le f(1+c,a)$, we need to check that $(1+c+c^2)(1+2a) \le (2+c)(1+a+ac)$. Equivalently,
     \[ 1+c+c^2+2a+2ac+2ac^2 \le 2+2a+3ac+c+ac^2 ,\]
     which after cancellations becomes $c^2(1+a) \le 1+ac$, which is trivially true since $c \le 1$.  
\end{proof}

\begin{proof}[Proof of \cref{lem:S-abc-iteration}]
    Recalling~\eqref{eq:psi_1}, we have
    \[ \psi_1(x) = \left(\frac{1+x_1+x_1x_2}{1+2x_1}\right)^d \le \left(\frac{1+x_1+cx_1}{1+2x_1}\right)^d \le \left(\frac{1+a+ac}{1+2a}\right)^d .\]
where the second inequality follows from \eqref{eq:f'}.
    Similarly
    \[ \psi_1(x) = \left(\frac{1+x_1+x_1x_2}{1+2x_1}\right)^d \ge \left(\frac{1+x_1}{1+2x_1}\right)^d \ge \left(\frac{1+b}{1+2b}\right)^d .\]
    
    Fix $n \ge 2$. Recall~\eqref{eq:psi_n}. Since $\psi_n(x)$ is clearly increasing in $x_{n+1}$, we have
    \[ \psi_n(x) \le x_{n-1}^d \left(\frac{1+x_n+cx_n}{1+x_{n-1}+x_{n-1}x_n}\right)^d .\]
    Since the latter expression is increasing in $x_{n-1}$ (this can be seen easily by dividing both the numerator and denominator by $x_{n-1}^2$) and $b \ge c$  so that $x_{n-1} \le b$, we have
    \[ \psi_n(x) \le b^d \left(\frac{1+x_n+cx_n}{1+b+bx_n}\right)^d .\]
    Since this expression is increasing in $x_n$ (follows from \eqref{eq:f'}), we have
    \[ \psi_n(x) \le b^d \left(\frac{1+c+c^2}{1+b+bc}\right)^d .\]
    thereby completing the proof.
\end{proof}

\subsection{Non-convergence for $d \ge 8$}\label{sec:nonconvergence}

\begin{prop}\label{prop:apriori_3}
    Fix $d \ge 8$, $a = .4, c = .01, a' = (.78)^d$. Then
    \[ \psi \circ \psi(\cS_{a,1,c}) \subset \cS_{a,1,c} \qquad\text{and}\qquad \psi(\cS_{a,1,c}) \cap \cS_{a',1,1} = \emptyset .\]
\end{prop}
\begin{proof}
    Denote
$$
c' = \left(\frac{1+c+c^2}{2+c}\right)^d ,
$$
    and recall from \cref{lem:S-abc-iteration} that
    \[ \psi(\cS_{0,1,c}) \subset \cS_{0,1,c'} .\]
    Note that $c'$ is decreasing in $d$. Therefore, $c' \le (c')^{8/d}=.0040677<c$, and hence,
    \[ \psi(\cS_{0,1,c}) \subset \cS_{0,1,c} .\]
    
     Pick $x  \in \cS_{a,1,c}$ and let $y = \psi(x)$ and $z = \psi(y)$. Then $\psi(\cS_{a,1,c}) \subset \psi(\cS_{0,1,c}) \subset \cS_{0,1,c}$. Therefore, $y_n \le c$ and $z_n \le c$ for all $n \ge 2$. Now applying these bounds, using that $x \mapsto \frac{(1+\beta x)^d}{(1+2x)^d}$ is decreasing in $x$ if $\beta<2$  and $x_1 \ge a$,
     \[
         y_1 = \frac{(1+(1+x_2)x_1)^d}{(1+2x_1)^d}\le \frac{(1+1.01x_1)^d}{(1+2x_1)^d}\le \frac{(1+1.01a)^d}{(1+2a)^d} = (.78)^d = a'.
     \]
     Now observe that 
     \begin{equation}
         z_1 =\frac{(1+(1+y_2)y_1)^d}{(1+2y_1)^d} \ge \frac{(1+y_1)^d}{(1+2y_1)^d} \ge \frac{(1+(.78)^d)^d}{(1+2(.78)^d)^d} =: \gamma(d).
     \end{equation}
    For $d=8,9$, we compute $\gamma(d)$ directly and see that
     \[ \gamma(8)  = .402\cdots >a \qquad\text{and}\qquad \gamma(9) = .436\cdots > a .\]
     For $d \ge 10$, we use the bound 
     $$
     \gamma(d) = \left(1+\frac{(.78)^d}{1+(.78)^d}\right)^{-d} \ge \exp(-d(.78)^d) \ge \exp(-10(.78)^{10}) = 0.434\cdots > a,
     $$
     where we used that $x(.78)^x$ is decreasing for $x > 1/\log(1/.78) \approx 4.025$.     
\end{proof}

\begin{thm}
Fix $d \ge 8$.
    The sequence $\psi^{(k)}(\zero)$ does not converge pointwise.
\end{thm}
\begin{proof}
    This is immediate from the previous proposition as it implies that $\psi^{(2k)}(\zero)_1 \ge .4$ and $\psi^{(2k+1)}(\zero)_1 \le (0.78)^d \le 0.14$ for all $k \ge 0$.
\end{proof}

\subsection{Convergence for $2 \le d \le 7$}\label{sec:convergence}

We shall prove a generalization of the convergence result stated in \Cref{thm:main_lipschitz}.
We say a sequence of weights $w=(w_i) \in [0,\infty)^\Z$ is \textbf{good} if it is symmetric (i.e. $w(i) = w(-i)$ for all $i \in \Z$), $w(0)>0$ and there exists a $k \in \N \cup \{0\}$ and a $c\in [0,1)$ such that
\[ w(i) = w(i+1) \quad\text{for all }0 \le i<k,\qquad\text{and}\qquad w(i+1) \le c \cdot w(i) \quad\text{for all }i \ge k .\]
In other words, the distribution is flat in a symmetric interval around 0, and decays at least at an exponential rate $c$ outside this interval. Let $\mu_{n,d}^w$ be the probability distribution on Lipschitz functions on $T_{n,d}$ where each Lipschitz function $f$ is weighted by the product of $w(f(v))$ over the leaves $v$. Note that $\mu_{n,d}^{\{-k,\ldots, k\}}$ is a special case of such a measure. 

\begin{thm}\label{thm:main_conv2}
   Fix $2 \le d \le 7$. Then $\lim_{n \to \infty} \mu_{n,d}^w$ exists for all good weight sequences $w$ and is independent of~$w$. In particular, $\lim_{n \to \infty} \mu_{n,d}^{\{-k,\ldots, k\}}$ exists for all $k \ge 0$ and is independent of~$k$.
\end{thm}

Note that $\mu_{n,d}^w$ is unaffected by a scaling of $w$. Thus, we may assume henceforth that $w$ is itself a probability distribution. Note in particular that this means it belongs to $\cE$ (it is symmetric and its support is an interval or all of $\Z$).
Our broad goal is to prove that the iterates of $\psi$ starting from $\varphi(w)$ converge pointwise. Along the way we also argue that all iterates stay in the set $\cS_{0,1,c}$, whose preimage under $\varphi$ is a tight family of distributions.
Consequently, we argue that $f_n(\rho)$ converges in distribution and that $\mu^0_{n,d}$ converges weakly.


The major step is to prove the following.
\begin{prop}\label{prop:psi_conv}
Fix $2 \le d \le 7$. Let $w$ be a good sequence of weights normalized to be a probability distribution. There exists $c \in (0,1)$ such that the $\psi^{(k)}({\mathsf R}(w))$ belongs to $\cS_{0,1,c}$ for all $k$ large enough and converges pointwise as $k \to \infty$.
\end{prop}

\begin{proof}[Proof of \Cref{thm:main_conv2} assuming \Cref{prop:psi_conv}]
It is not hard to check that ${\mathsf R}^{-1}$ is continuous on its domain $\cR$ (when both the domain and range are endowed with the pointwise topology).
    Since $F^{(k)}(w) = {\mathsf R}^{-1} \circ \psi^{(k)} \circ {\mathsf R}(w)$, we conclude from \Cref{prop:psi_conv} that $F^{(k)}(w)$ converges pointwise and that $F^{(k)}(w)$ belongs to ${\mathsf R}^{-1}(\cS_{0,1,c})$ for all $k$ large enough. Observe that ${\mathsf R}^{-1}(\cS_{0,1,c})$ is a tight family of distributions.
    Thus, letting $f_{n} \sim \mu_{n,d}^w$ and recalling \cref{lem:F} (or rather its straightforward extension from $\mu^0_{n,d}$ to $\mu^w_{n,d}$) , this implies that $f_n(\rho)$ converges in distribution as $n \to \infty$.

    It is now not too hard to see that $f_n$ converges locally. We very briefly give the idea. Fix $r \ge 0$ and consider the restriction $f_{n,r}$ of $f_n$ to $L_r$ (assume $n$ large). Then, for a given Lipschitz function $\xi$ on $L_r$, we have that $\P(f_{n,r}=\xi)$ is proportional $\prod_v \P(f_{n-r}(\rho)=\xi(v))$, where the product is over the leaves of $L_r$. Since this product converges as $n\to\infty$, we see that the distribution of $f_{n,r}$ converges as $n\to\infty$, which shows that $f_n$ converges locally.
\end{proof}

To prove \Cref{prop:psi_conv} we employ the following strategy.
First, we shall show that
\begin{equation}\label{eq:good}
 \text{if $w$ is good then there exist $k \ge 0$ and $c \in [0,1)$ such that $\psi^{(k)}({\mathsf R}(w)) \in \cS_{0,1,c}$.}
\end{equation}
This reduces the problem to understanding the iterations of $\psi$ starting from an initial point in $\cS_{0,1,c}$.
For this we will show that there exist numbers $a_*,b_*,c_*$ with $0 \le a_*,c_* \le b_* \le 1$ such that
\begin{equation}\label{eq:S-iteration}
 \text{for every $0 \le c < 1$ there exists $k$ such that $\psi^{(k)}(\cS_{0,1,c}) \subseteq \cS_{a_*,b_*,c_*}$}, 
\end{equation}
and
\begin{equation}\label{eq:contraction}
 \text{$\psi$ preserves the convex set $\cS_{a_*,b_*,c_*}$ and is a contraction on it}.
\end{equation}
The contraction referred to in~\eqref{eq:contraction} is with respect to the metric induced by the norm given by 
\begin{equation}\label{eq:norm}
\|x\| := \begin{cases} \sup_{i \ge 1} |x_i| &\text{if }d=2 \\ |x_1| + \sup_{i \ge 2} |x_i| &\text{if }3 \le d \le 7 \end{cases}.
\end{equation}
Let us remark that the two norms are equivalent (of course, in the case of $d=2$, this is just the usual $\ell^\infty$-norm). However, the map $\psi$ is not a contraction with respect to the $\ell^\infty$ norm when $d \ge 4$ even very near the fixed point (see \Cref{table:abc_derivative_approx}), and for this reason we use a slightly modified norm (we use this norm also for $d=3$). Let us mention that the $\ell^1$-norm is also a natural candidate, but that this would not allow to obtain the full stated result as $\cS_{0,1,c}$ is not in $\ell^1$.

Let us show how the above easily yields \Cref{prop:psi_conv}.
\begin{proof}[Proof of \Cref{prop:psi_conv} assuming~\eqref{eq:good}-\eqref{eq:contraction}]
Let $w$ be a good sequence of weights normalized to be a probability distribution. Denote $x^{(k)} := \psi^{(k)}(R(w))$. By~\eqref{eq:good}, there exists $k_0$ and $c$ such that $x^{(k_0)} \in \cS_{0,1,c}$. By~\eqref{eq:S-iteration}, there exists $k_1$ such that $x^{(k_0+k_1)} \in \cS_{a_*,b_*,c_*}$. By~\eqref{eq:contraction}, there exists $\lambda \in (0,1)$ such that $x^{(k)} \in \cS_{a_*,b_*,c_*}$ and $\|x^{(k+1)}-x^{(k)}\| \le \lambda\|x^{(k)}-x^{(k-1)}\|$ for all $k \ge k_0+k_1$. Since $\cS_{a_*,b_*,c_*} \subset \cS_{0,1,c_*}$, we have that $x^{(k)} \in \cS_{0,1,c_*}$ for all $k$ large enough, and since $\|\cdot\|$ makes $\cS_{0,1,c_*}$ a complete metric space, we have that $x^{(k)}$ converges in this space, which clearly implies pointwise convergence.
\end{proof}

The next two sections are devoted to the proofs of~\eqref{eq:good}-\eqref{eq:contraction}. In \cref{sec:apriori_bounds} we prove~\eqref{eq:good} and~\eqref{eq:S-iteration} and in \cref{sec:contraction} we prove~\eqref{eq:contraction}.

\subsection{Getting absorbed into the basin of attraction}\label{sec:apriori_bounds}

\begin{proof}[Proof of \eqref{eq:good}]
    For $k \ge 0$, let $\cA_k$ be the set of $x \in \cR$ such that $x_i=1$ for $i \le k$ and $\sup_{i>k} x_i < 1$.
    Note that $\cA_0 \cup \cA_1 = \bigcup_{c \in [0,1)} \cS_{0,1,c}$.
    Let $w$ be a good sequence of weights normalized to be a probability distribution. Note that $R(w) \in \bigcup_{k \ge 0} \cA_k$. Thus, it suffices to show that
    \[ \psi(\cA_k) \subset \cA_{k-1} \qquad\text{for all }k \ge 2 .\]
    Fix $k \ge 1$ and $x \in \cA_k$. Denote $c := \sup_{i>k} x_i < 1$. It is clear from~\eqref{eq:psi_1} and~\eqref{eq:psi_n} that $\psi_i(x)=1$ for $i\le k-1$ and that $\psi_i(x)<1$ for $i \ge k$. Moreover, in the same manner as in the proof of \cref{lem:S-abc-iteration}, one obtains a uniform upper bound on $\psi_i(x)<1$ for $i \ge k$, namely,
    \[ \psi_i(x) \le \left(\frac{1+c+c^2}{2+c}\right)^d .\]
    This shows that $\psi(x) \in \cA_{k-1}$.
\end{proof}

We now move on to the proof of~\eqref{eq:S-iteration}.
Given \cref{lem:S-abc-iteration}, the proof of \eqref{eq:S-iteration} boils down to understanding the behavior of the iterations of $\varphi$ applied to $\cS_{0,1,c}$. It is not too hard to show that the iterations of $\varphi$ do converge to a fixed point $(a_*,b_*,c_*)$ if we start from $a=0,b=1,c>1/2$ (this part is true for all values of $d$). Nevertheless to have estimates of the derivative of $\psi$, we need good estimates on $(a_*,b_*,c_*)$. Iterations in a computer yield the estimates of $(a_*,b_*,c_*)$ laid out in \Cref{table:abc_computer}.

\begin{table}[h]
    \centering
    \begin{tabular}{|c|c|c|c|c|c|c|c|}
    \hline
  $d$  &2 &3&4&5&6&7&8  \\
  \hline
    $a_*$ &.5192 &.4374 & .3762& .3294 &.2932 & .2645 & .1027 \\ \hline
    $b_*$ &.6335 &.4649 &.3828 &.3310 &.2935 & .2646 & .4906 \\ \hline
    $c_*$ & .1988 & .0344 & .0060& $9.5\cdot 10^{-4}$& $1.4 \cdot 10^{-4}$& $1.8\cdot 10^{-5}$ & $1.4 \cdot 10^{-4}$ \\ \hline
\end{tabular}
\caption{The fixed point of $\varphi$. Simulated in computer by running ten thousand iterations started from $a=0,b=1,c=.9$.}\label{table:abc_computer}
\end{table}
For $d \ge 8$, iterations show that $a_*$ and $b_*$ are far apart, which is a manifestation of the non-convergence (in reality, as we have seen in \cref{sec:nonconvergence}, the actual value of $x_1$ alternates between being close to $a_*$ and $b_*$).

To make the result completely rigorous, we obtain rigorous estimates of $a_*,b_*,c_*$, outlined in \Cref{table:abc_final_section2}. While we could have obtained much better estimates (closer to the actual fixed point), it will turn out that these estimates are enough to bound the operator norm of the derivative of $\psi$ (with respect to the norm \eqref{eq:norm} on the base space). For larger values of $d$ (especially $d=7$), we need slightly better precision.

\begin{prop}\label{prop:estimate_fixed_point}
Fix $2 \le d\le 7$. Let $(\hat a_*,\hat b_*,\hat c_*)$ be as in \Cref{table:abc_final_section2}. Then, for any $0 \le c<1$, $\psi^{(k)}(\cS_{0,1,c}) \subset \cS_{\hat a_*,\hat b_*,\hat c_*}$ for all large enough $k$.
   \begin{table}
    \centering
    \begin{tabular}{|c|c|c|c|c|c|c|}
    \hline
  $d$  &2 &3&4&5&6&7  \\
  \hline
    $\hat a_*$ &.5 &.4 &.3 &.3 &.27 &.26 \\ \hline
    $\hat b_*$ &.7 &.6 &.5 &.4 &.32 &.27 \\ \hline
    $\hat c_*$ &.27 &.2 &.1 &.1 &.1 &.01 \\ \hline
\end{tabular}
\caption{Rigorous estimates of $a_*,b_*,c_*$. Proofs can be found in \Cref{sec:phi_estimate}.}\label{table:abc_final_section2}
\end{table} 
\end{prop}
The proof of \Cref{prop:estimate_fixed_point} is rather technical, so we postpone it to \Cref{sec:phi_estimate} on order not to disrupt the flow of the argument. Readers who are content with the computer estimates in \Cref{table:abc_computer} can skip \Cref{sec:phi_estimate} altogether.


\subsection{Contraction via the derivatives}\label{sec:contraction}

\begin{table}
    \centering
    \scalebox{1.1}{
    \begin{tabular}{|c|c|c|c|c|c|c|}
    \hline
  $d$  &2&3&4&5&6&7  \\
  \hline
   $x_1$ & .5992 & .4555 & .3803 & .3303 & .2936 & .2646 \\ \hline
   $x_2$ & .1712 & .0327 & .0059 & $9.5 \cdot 10^{-4}$ & $1.4 \cdot 10^{-4}$ & $1.8 \cdot 10^{-5}$ \\ \hline
   $x_3$ & .0222 & $3.2 \cdot 10^{-5}$ & $1.2 \cdot 10^{-9}$ & $1.6 \cdot 10^{-16}$ & $6.4 \cdot 10^{-24}$ & $5.1 \cdot 10^{-34}$ \\ \hline
   $x_4$ & $4.7 \cdot 10^{-4}$ & $3.2 \cdot 10^{-14}$ & $1.8 \cdot 10^{-36}$ & $2.5 \cdot 10^{-76}$ & $6.8 \cdot 10^{-140}$ & $9.1 \cdot 10^{-234}$ \\ \hline
   $-\frac{\partial \psi_1}{\partial x_1}$ &.2655  & .4704 & .6213 & .7467 & .8575 & .9577 \\ \hline
   $\frac{\partial \psi_1}{\partial x_2}$  &.4220  & .4234 & .4184 & .4100 & .3992 & .3874 \\ \hline
   $\frac{\partial \psi_2}{\partial x_1}$  &.3357  & .1467 & $1.8 \cdot 10^{-7}$ & .0108 & .0022 & $3.7 \cdot 10^{-4}$ \\  \hline
   $\frac{\partial \psi_2}{\partial x_2}$  &.1773  & .0647 & $4.6 \cdot 10^{-9}$ & .0036 & $6.3 \cdot 10^{-4}$ & $9.7 \cdot 10^{-5}$ \\ \hline
\end{tabular}
}
\caption{The fixed point $x_*$ and some partial derivatives of $\psi$ at this point. Simulated in computer by running a million iterations.}\label{table:abc_derivative_approx}
\end{table}

Our goal in this section is to prove~\eqref{eq:contraction}. We do this by bounding the derivatives of $\psi$.
Recall the norm $\|\cdot\|$ defined in~\eqref{eq:norm}.
It is straightforward to check that the Frech\'et derivative of the map $\psi$ evaluated at $x$ is the linear operator on the vector space spanned by $\cR$ given by
$$
(D\psi_x(y))_i = (D\psi_i)_x(y) = \sum_{j \ge 1} \frac{\partial \psi_i}{\partial x_j} \cdot y_j.
$$
Note that $\partial \psi_i / \partial \psi_j = 0$ unless $|i-j| \le 1$. 
Recall that the operator norm of $D\psi_x$ is $\|D\psi_x\| := \sup_{y : \|y\|=1} \|D\psi_x y\|$. Here the norm we use is as in \Cref{eq:norm}.

\begin{prop}\label{prop:op_norm_estimate}
    $\|D\psi_x\| <.99$ for all $x \in \cS_{a,b,c}$ with $(a,b,c)$ as in \Cref{table:abc_final_section2}.
\end{prop}

From here it is simple to deduce~\eqref{eq:contraction}.
\begin{proof}[Proof of \eqref{eq:contraction} assuming \cref{prop:op_norm_estimate}]
    Let $(a,b,c)$ be as in \Cref{table:abc_final_section2}.
    The fact that $\psi$ preserves $\cS_{a,b,c}$ is the content of \Cref{prop:estimate_fixed_point}. Let $x,y\in \mathcal{S}_{a,b,c}$, and let $T$ denote a bounded linear functional on the vector space spanned by $\mathcal{R}$,  such that $\|T\|=1$ (the operator norm is defined with respect to the norm given in \eqref{eq:norm}). Since $\psi$ is differentiable, via the chain rule we have that 
    \[
    D(T\circ \psi)=T\circ D\psi. 
    \]
    For $t \in [0,1]$, define
    \[
    x(t):=ty+(1-t)x\qquad\text{and}\qquad r(t):=T\circ \psi(x(t)).
    \]
By the fundamental theorem of calculus,
\begin{align*}
T\circ \psi(y)- T\circ \psi(x) &=\int_{0}^{1}r'(t)dt \\ 
&=\int_{0}^{1} T\circ D\psi_{x(t)}\left(y-x\right)dt.   
\end{align*}
Thus, using that $\|T\|=1$,
\begin{align*}
|T\circ\psi(y)-T\circ\psi(x)|&\leq \int_{0}^{1}\|T\circ D\psi_{x(t)}\| \cdot \|y-x\|dt \\ 
&\leq  \sup_{t\in [0,1]} \|D\psi_{x(t)}\| \cdot \|y-x\|.
\end{align*}
Since $\mathcal{S}_{a,b,c}$ is convex, $x(t)\in \mathcal{S}_{a,b,c}$ for all $t\in [0,1]$, so we may replace the upper bound of the norm of the derivative on the path $x(t)$ with the upper bound of the same on $\mathcal{S}_{a,b,c}$. Finally, taking the supremum over all $T$ such that $\|T\|=1$, we may conclude that
    \begin{equation*}
        \|\psi(x)  - \psi(y)\| \le \sup_{y \in \cS_{a,b,c}}\|D\psi_y\| \cdot \|x-y\|,
    \end{equation*}
  which combined with \Cref{prop:op_norm_estimate} gives that $\psi$ is a contraction on $\cS_{a,b,c}$.
\end{proof}

The rest of this section is devoted to the proof of \cref{prop:op_norm_estimate}.
We begin by giving an upper bound for the operator norm of $D\psi_x$ in terms of the partial derivatives of $\psi$.

\begin{lemma}\label{lem:op_norm}
For $d=2$, the operator norm of $D\psi_x$ is
\begin{align*}
 \|D\psi_x\|
  &= \sup_{i \ge 1} \sum_{j \ge 1} \left|\frac{\partial \psi_i}{\partial x_j}\right| \\
  &= \max\left\{ \left|\frac{\partial \psi_1}{\partial x_1}\right| + \left|\frac{\partial \psi_1}{\partial x_2}\right| ,~ \sup_{i \ge 2} \left(\left|\frac{\partial \psi_i}{\partial x_{i-1}}\right|+\left|\frac{\partial \psi_i}{\partial x_i}\right|+\left|\frac{\partial \psi_i}{\partial x_{i+1}}\right|\right) \right\}.
\end{align*}
For $d \ge 3$, the operator norm of $D\psi_x$ is
\begin{align*}
 \|D\psi_x\|
  &\le \max\left\{ \left|\frac{\partial \psi_1}{\partial x_1}\right| + \sup_{i \ge 2} \left|\frac{\partial \psi_i}{\partial x_1}\right| ,~ \sum_{j \ge 2} \left|\frac{\partial \psi_1}{\partial x_j}\right| + \sup_{i \ge 2} \sum_{j \ge 2} \left|\frac{\partial \psi_i}{\partial x_j}\right| \right\} \\
  &= \max\left\{ \left|\frac{\partial \psi_1}{\partial x_1}\right| + \left|\frac{\partial \psi_2}{\partial x_1}\right| ,~ \left|\frac{\partial \psi_1}{\partial x_2}\right| + \sup_{i \ge 2} \sum_{j=i-1}^{i+1} \left|\frac{\partial \psi_i}{\partial x_j}\right| \1_{j \ge 2} \right\}.
\end{align*}
\end{lemma}
\begin{proof}
    We only prove the case of $d \ge 3$ (the case $d=2$ is similar and more standard as the norm is simply the $\ell^\infty$-norm).
    Denote $T := D\psi_x$ and $t_{i,j} := \frac{\partial \psi_i}{\partial x_j}$, so that $(Ty)_i = \sum_j t_{i,j} y_j$. Denote $m_i := \sum_{j \ge 2} |t_{i,j}|$.
    Let $y$ have norm 1. Write $a:=|y_1|$ and $b:=1-|a|$, and note that $|y_i| \le b$ for all $i \ge 2$.
    Then
    \[ |(Ty)_i| = \left|\sum_j t_{i,j} y_j\right| \le |t_{i,1}y_1| + \sum_{j \ge 2} |t_{i,j} y_j| \le a|t_{i,1}| + b m_i .\]
    Hence,
    \begin{align*}
     \|Ty\| = |(Ty)_1| + \sup_{i \ge 2} |(Ty)_i|
      &\le a|t_{1,1}| + b m_1 + \sup_{i \ge 2} (a|t_{i,1}| + b m_i) \\
      &\le \max\left\{ |t_{1,1}|+\sup_{i \ge 2} |t_{i,1}|,~ m_1 + \sup_{i \ge 2} m_i \right\} ,
    \end{align*}
    as desired.
\end{proof}

We now record the expressions for the derivatives.
\begin{align}
    -\frac{\partial \psi_1}{\partial x_1} &= \frac{d(1-x_2)(1+x_1+x_1x_2)^{d-1}}{(1+2x_1)^{d+1}}, \label{eq:dpsi_1_x_1} \\
    \frac{\partial \psi_1}{\partial x_2} & = \frac{dx_1(1+x_1+x_1x_2)^{d-1}}{(1+2x_1)^d} , \label{eq:dpsi_1_x_2}
    \end{align}
    and for $n \ge 2$,
    \begin{align}
    \frac{\partial \psi_n}{\partial x_{n-1}} &= \frac{dx_{n-1}^{d-1}(1+x_n+x_nx_{n+1})^d}{(1+x_{n-1}+x_{n-1}x_n)^{d+1}} ,\label{eq:dpsi_n_x_n-1} \\
    \frac{\partial \psi_n}{\partial x_n} &= \frac{dx_{n-1}^{d}(1+x_n+x_nx_{n+1})^{d-1}(1+x_{n-1}x_{n+1}+x_{n+1})}{(1+x_{n-1}+x_{n-1}x_n)^{d+1}}, \label{eq:dpsi_n_x_n}\\
    \frac{\partial \psi_n}{\partial x_{n+1}} &= \frac{dx_{n-1}^d x_n(1+x_n+x_nx_{n+1})^{d-1}}{(1+x_{n-1}+x_{n-1}x_n)^d}. \label{eq:dpsi_n_x_n+1}
\end{align}

The next two lemmas provide bounds on the partial derivatives of $\psi$. While we assume throughout that $(a,b,c)$ is chosen according to \Cref{table:abc_final_section2}, we use only minor information about the actual values, and keep bounds general where possible. We shall only plug in the exact values of these parameters at the very end. The simple properties we need are that $c < 1/2$ and that
\begin{equation}\label{eq:abc_prop1}
    (d-2)a \le 1\quad\text{for }2 \le d \le 5,\qquad(d-2)a \ge 1+dac\quad\text{for }d=6,7.
\end{equation}
\begin{equation}\label{eq:abc_prop2}
    b((1-c)d-2) \le 1\quad\text{for }d=2,3,4 .
\end{equation}
\begin{equation}\label{eq:abc_prop3}
    1+2b(1+c) \le d\quad\text{for }d \ge 3 .
\end{equation}

\begin{lemma}\label{lem:derivative12}
Let $x \in \cS_{a,b,c}$ with $(a,b,c)$ as in \Cref{table:abc_final_section2}. 
 Then
\begin{align}
    \left|\frac{\partial \psi_1}{\partial x_1}\right| &\le \frac{d(1+a)^{d-1}}{(1+2a)^{d+1}} &&\text{for $2 \le d \le 5$.}\\
     \left|\frac{\partial \psi_1}{\partial x_1}\right|    & \le  \frac{d(1-c)(1+a+ac)^{d-1}}{(1+2a)^{d+1}} &&\text{for $d = 6,7$.}\label{psi1x1}\\
    \left|\frac{\partial \psi_1}{\partial x_2}\right| & \le \frac{db(1+b+bc)^{d-1}}{(1+2b)^d} &&\text{for $2 \le d \le 4$}.\\
    \left|\frac{\partial \psi_1}{\partial x_2}\right| & \le \frac{(d-1)^{d-1}}{d^{d-1}(1-c)} &&\text{for $5 \le d \le 7$.} \label{psi1x2}
    \end{align}
    \end{lemma}
    
\begin{proof}
Recall the expression for $|\partial \psi_1/\partial x_1|$ from~\eqref{eq:dpsi_1_x_1}. Since this expression is decreasing in $x_1$ (using \eqref{eq:f'}) we have
    \begin{equation}\label{eq:dpsi_1_x_1_bound}
    \left|\frac{\partial \psi_1}{\partial x_1}\right| \le \frac{d(1-x_2)(1+a+ax_2)^{d-1}}{(1+2a)^{d+1}} =: h(x_2) .
    \end{equation}
The function $h(x)$ is maximized at $$x_* := \frac{(d-2)a-1}{da}.$$
  By~\eqref{eq:abc_prop1}, we have that $x_*<0$ for $2 \le d \le 5$ and $x_*>c$ for $d=6,7$.
  Note that $h$ is increasing on $(-\infty,x_*]$ and decreasing on $[x_*,\infty)$.
  Overall, $h$ is decreasing on $[0,c]$ for $2 \le d \le 5$ and increasing on $[0,c]$ for $d=6,7$. 
    Thus, $|\partial \psi_1/\partial x_1| \le h(0)$ for $2 \le d \le 5$ and $|\partial \psi_1/\partial x_1| \le h(c)$ for $d=6,7$. This establishes the claimed bound on $|\partial \psi_1/\partial x_1|$.


    Next, recall the expression for $|\partial \psi_1/\partial x_2|$ from~\eqref{eq:dpsi_1_x_2}. Since this expression is increasing in $x_2$, 
    $$
    \left|\frac{\partial \psi_1}{\partial x_2}\right|  \le \frac{dx_1(1+x_1+cx_1)^{d-1}}{(1+2x_1)^d} =: \tilde h(x_1) .
    $$
Differentiating $\tilde h$, we get 
    $$
    \tilde h'(x) =\frac{d(1+x+cx)^{d-2}}{(1+2x)^{d+1}} [x(2-(1-c)d) +1].
    $$
 If $d=2$, then clearly $\tilde h'(x) \ge 0$ for $x \ge 0$, so that $\tilde h(x_1) \le \tilde h(b)$. For $d \ge 3$, we have a critical point at 
 $$
 \tilde x_* := \frac1{(1-c)d-2}.
 $$
 Note that $2-(1-c)d < 0$ since $c<1/2$. Thus, $\tilde x_*$ is positive and is in fact a maximum of $\tilde h$. Furthermore, $\tilde x_*>b$ for $d=3,4$ by~\eqref{eq:abc_prop2}. Thus, $\tilde h(x_1) \le \tilde h(b)$ for $d \le 4$, and $\tilde h(x_1) \le \tilde h(\tilde x_*)$ for $d \ge 5$. This establishes the claimed bound on $|\partial \psi_1/\partial x_2|$.
\end{proof}

\begin{lemma}\label{lem:derivative_nge2}
    Let $x \in \cS_{a,b,c}$ with $(a,b,c)$ as in \Cref{table:abc_final_section2}.
For $n \ge 2$,
    \begin{align}
    \left|\frac{\partial \psi_n}{\partial x_{n-1}}\right| &\le \frac{4d(d-1)^{d-1}(1+c+c^2)^d}{(d+1)^{d+1}(1+c)^{d-1}} &&\text{for $d=2$},\label{psinxnminus1_d2}\\
    \left|\frac{\partial \psi_n}{\partial x_{n-1}}\right| & \le \frac{db^{d-1}(1+c+c^2)^d}{(1+b)(1+b+bc)^{d}} &&\text{for $3 \le d \le 7$,}\\
    \left|\frac{\partial \psi_n}{\partial x_n}\right| & \le \left|\frac{\partial \psi_n}{\partial x_{n-1}}\right| \cdot  b(1+c+bc).\label{psinxn} \\
    \left|\frac{\partial \psi_n}{\partial x_{n+1}}\right| &\le \frac{db^d c(1+c+c^2)^{d-1}}{(1+b+bc)^d}.\label{psinxnplus1}
\end{align}
\end{lemma}  

We remark that it is easy to show that any $x \in \psi(\cS_{a,b,c})$ satisfies $x_n \le c^d$ for $n \ge 3$, and that one could use this extra information to improve the bounds given in the lemma. This did not seem, however, to simplify other parts of the proof and so we did not find any reason to do this. Furthermore, the bounds we stated are uniform in $n \ge 2$ and governed by the case $n=2$; the bounds can be significantly improved as $n$ increases.

\begin{proof}
    Set $\beta := 1+x_n$, and for $x \ge 0$, define
    $$
    h(x) := \frac{x^{d-1}}{(1+\beta x)^{d+1}} ,
    $$
    so that
    \[ \left|\frac{\partial \psi_n}{\partial x_{n-1}}\right| = h(x_{n-1}) \cdot d(1+x_n+x_nx_{n+1})^d .\]
    Note that $h$ is maximized at $x_* := \frac{d-1}{2\beta}$, and that it is increasing on $[0,x_*]$ and decreasing on $[x_*,\infty)$.
    For $d \ge 3$, we have that $x_{n-1} \le b \le x_*$ by~\eqref{eq:abc_prop3}, and hence, $h(x_{n-1}) \le h(b)$ and
    \begin{align}
        \left|\frac{\partial \psi_n}{\partial x_{n-1}}\right| & \le \frac{db^{d-1}(1+x_n+x_nx_{n+1})^d}{(1+b+bx_n)^{d+1}} \\
        & \le \frac{db^{d-1}(1+c+c^2)^d}{(1+b)(1+b+bc)^{d}},
    \end{align}
    where we used $1+b+bx_n \ge 1+b$, $x_{n+1} \le c$ and \eqref{eq:f'}.
    For $d=2$, we do not know whether $x_* \ge b$, so instead we just use that $h(x_{n-1}) \le h(x_*)$ and plug in the value of $x_*$ to get that
    \begin{align}
        \left|\frac{\partial \psi_n}{\partial x_{n-1}}\right| & \le \frac{d\left(\frac{d-1}{2(1+x_n)}\right)^{d-1}(1+x_n+x_nx_{n+1})^d}{\left(1+\frac{d-1}{2}\right)^{d+1}}\\
        & = \frac{4d(d-1)^{d-1} (1+x_n+x_nx_{n+1})^d}{(d+1)^{d+1} (1+x_n)^{d-1}}\nonumber\\
        & \le \frac{4d(d-1)^{d-1} (1+c+c^2)^d}{(d+1)^{d+1} (1+c)^{d-1}},
    \end{align}
    where we used $x_{n+1} \le c$ and \eqref{eq:f'}.

    To upper bound $\partial \psi_n/\partial x_n$, note that
\[ \left|\frac{\partial \psi_n}{\partial x_n}\right| = \left|\frac{\partial \psi_n}{\partial x_{n-1}}\right| \cdot \frac{x_{n-1}(x_{n-1}x_{n+1} + x_{n+1}+1)}{x_nx_{n+1}+x_n+1} ,\]
    and then, for the second term on the right-hand side, lower bound the denominator by 1 and upper bound the numerator using $x_{n-1} \le b$ and $x_{n+1} \le c$. Finally, 
    \begin{align}
        \left|\frac{\partial \psi_n}{\partial x_{n+1}}\right| &= \frac{dx_{n-1}^d x_n(1+x_n+x_nx_{n+1})^{d-1}}{(1+x_{n-1}+x_{n-1}x_n)^d} \\
        & \le \frac{db^d x_n(1+x_n+x_nx_{n+1})^{d-1}}{(1+b+bx_n)^d}\\
        & \le \frac{db^d c(1+c+c^2)^{d-1}}{(1+b+bc)^d},
    \end{align}
    where in the first inequality we used that $|\partial \psi_n/\partial x_{n+1}|$ is increasing in $x_{n-1}$ and that $x_{n-1} \le b$, and in the second inequality we used $x_n,x_{n+1} \le c$ and \eqref{eq:f'}.
\end{proof}

\begin{table}[h]
    \centering
    \scalebox{1.1}{
    \begin{tabular}{|c|c|c|c|c|c|c|}
    \hline
  $d$   &2&3&4&5&6&7  \\
  \hline
   $\frac{\partial \psi_1}{\partial x_1}$       &.38  & .57 &  .84 & .86 & .97 & .986 \\ \hline
   $\frac{\partial \psi_1}{\partial x_2}$       &.46  & .51 &  .47 & .46 & .45 & .401 \\ \hline
   $\frac{\partial \psi_n}{\partial x_{n-1}}$   &.43  & .26 &  .09 & .03 & .01 & .001 \\ \hline
   $\frac{\partial \psi_n}{\partial x_{n}}$     &.43  & .21 & .06 & .02  & .01 & .001 \\ \hline
   $\frac{\partial \psi_n}{\partial x_{n+1}}$   &.10  & .04  & .01  & .01 & .01 & .001 \\ \hline
   $\|D\psi_x\|$ & .96 & .83 & .95 & .89 & .98 & .987 \\ \hline
\end{tabular}
}
\caption{Upper bounds on the partial derivatives and operator norm when $x \in \cS_{a,b,c}$ with $(a,b,c)$ as in \Cref{table:abc_final_section2}. The bounds on the partial derivatives are obtained by plugging in the values of $a,b,c$ into the bounds given by \Cref{lem:derivative12,lem:derivative_nge2} ($n \ge 2$ in the table) and rounding up. The bounds on the operator norm are then obtained by plugging in the former bounds into the bounds from \cref{lem:op_norm}.}\label{table:abc_derivative_approx}
\end{table}

\begin{proof}[Proof of \cref{prop:op_norm_estimate}]
    Using \Cref{lem:derivative12,lem:derivative_nge2}, we get upper bounds on the partial derivatives $|\partial \psi_i/\partial x_j|$, which when combined with \Cref{lem:op_norm} yields an upper bound on the operator norm $\|D\psi_x\|$ (see \Cref{table:abc_derivative_approx}).
\end{proof}
\begin{proof}[Proof of \Cref{thm:main_lipschitz}]
The proof of the nonconvergence part is in \Cref{sec:nonconvergence}. The convergence part follows from the more general \Cref{thm:main_conv2}.
\end{proof}




\section{Long-range order for graphs with large expansion}\label{sec:long_range}

In this section, we prove \Cref{thm:main_general}.
Recall that $G$ is an infinite bipartite connect graph with maximum degree $d$ and Cheeger constant $h$. We fix a finite subset $G'$ of $G$. Recall that $\nu^{a,b}_{G',M}$ is the uniform measure on all $M$-Lipschitz functions on $G' \cup \partial G'$ whose values on $\partial G'$ are constrained to be in $\{a,\ldots,b\}$ for white vertices or in $\{b-M,\dots,a+M\}$ for black vertices.

We now fix two integers $a,b$ with $0 \le b-a \le 2M$.  To facilitate the description of these Gibbs measures, we shall consider the following sets:
\begin{align*}
    \cO(f)&:= \{v \in \T_{2n}^d \text{ is odd}, f(v) \not \in \{ b-M, \ldots, a+M\}\},\\
    \cE(f)&:= \{v \in \T_{2n}^d \text{ is even}, f(v) \not \in \{a,\ldots,b\}\}.
\end{align*}
and define $B(f):= \cO(f) \cup \cE(f)$. Let $A(f,v)$ denote the distance two connected component of $B(f)$ containing $v$. To elaborate: connect two vertices if they are at distance at most two, and let $A(f,v)$ denote the connected component of $v$ thus formed.
Our goal in the next proposition will be to prove that the size of $A(f,v)$ has exponential tail once $h$ is much larger compared to $M$.

\begin{prop}\label{prop:cluster_exp_tail}
Assume that $h \ge 4M\log(d^4(4M+1))$ and let $f \sim \nu^{a,b}_{G',M}$. Then for all $v_0 \in G'$ and $n \ge 1$,
    $$
     \P(|A(f,v_0)|  =  n) \le \exp\left(-\frac{hn}{4M}\right).
    $$   
\end{prop}

 The upshot of \Cref{prop:cluster_exp_tail} is that most of the odd vertices take values in $\{b-M,\ldots,a+M\}$
and most of the even vertices take values in $\{a,\ldots,b\}$. 

Let us illustrate some special cases. If $a=b=0$ and $M=1$ (1-Lipschitz case with 0 boundary conditions on even vertices), most of the even vertices take the value 0 and the odd vertices take values in $\{-1,0,1\}$ (roughly uniformly at random as they are likely to be surrounded by 0s). Also, if $a=0,b=1$ (as in \Cref{thm:main_lipschitz_FKG}), then both even and odd vertices are likely to take values in $\{0,1\}$. In general, if $b-a = M$, the even and odd vertices play a symmetric role in the measure (and in fact bipartiteness of $G$ is not needed).

\begin{proof}[Proof of \Cref{prop:cluster_exp_tail}]
    We begin with an elementary fact (see, e.g., \cite[Lemma 1.8]{PWY_expander}). Let $\cL$ denote the set of $M$-Lipschitz functions on $G'$ satisfying the required boundary constraints. Let $\cE$ be a subset of $\cL$ and suppose that $T$ is a map from $\cE$ to subsets of $\cL$. Also assume that there exist $s,t>0$ such that $|T(g)| \ge t$ for all $g \in \cE$, and $|\{g \in \cA: g' \in T(g) \}| \le s$ for all $g' \in \cL$. Then $$ \P(f \in \cE) = \frac{|\cE|}{|\cL|} \le \frac st.$$

    Throughout, we write $A(f)$ for $A(f,v_0)$ and observe that $A(f)$ cannot contain any vertex of $\partial G'$ by definition. Fix a distance two connected set $A$ containing $v_0$ but not intersecting $\partial G'$ and let $\cE := \{g \in \cL: A(g)=A\}$. We shall construct a map $T$ as described in the previous paragraph so that $s/t$ is small, where $s,t$ are as in the previous paragraph. Consequently, $$\P(A(f) =A) = \P(f \in \cE) \le \frac st.$$ 
    Our goal is to make $\frac st$ exponentially small in the size of $A$ with a suitably large constant, so that we can then apply a union bound over all possible realizations of $A(f)$ and conclude.

    Let us describe the map $T$. For $g \in \cE$, let $T(g) $ be the set of $g' \in \cL$ satisfying
    \begin{equation}\label{eq:f'}
     g'(x) 
     \begin{cases}
         = g(x) &\text{if }x \not \in A \cup \partial A\\
         \in \{a,\dots,b\} &\text{if }x \in A \cup \partial A, x \text{ even,}\\
         \in \{ b-M,\dots,a+M\} &\text{if } x \in A \cup \partial A, x \text{ odd.}\\
     \end{cases}
    \end{equation}
    The point of this definition is that, as we will see, it allows much more possibilities in $A \cup \partial A$ for $g' \in T(g)$ compared to~$g$. Specifically, we claim $T$ satisfies the requirement with
    \begin{equation}\label{eq:bound_s}
        s := (4M+1)^{|A|} (b-a)^{|\{v \in  \partial A: v \text{ odd}|}(a-b+2M)^{|\{v \in  \partial A: v \text{ even}|},
    \end{equation}
    and
    \begin{equation}\label{eq:bound_t}
        t := (b-a+1)^{|\{v \in  A \cup \partial A: v \text{ odd}|}(a-b+2M+1)^{|\{v \in  A \cup \partial A: v \text{ even}|}.
    \end{equation}
    Before proving this claim, let us show how it yields the proposition.
    We have
  \begin{equation*}
      \frac{s}{t} \le (4M+1)^{|A|} \left(\frac{b-a}{b-a+1}\right)^{|\{v \in  \partial A, v \text{ odd}|}\left(\frac{a-b+2M}{a-b+2M+1}\right)^{|\{v \in  \partial A, v \text{ even}|}
  \end{equation*}
   Since $x \mapsto \frac{x}{x+1}$ is increasing and $0\le b-a \le 2M$, and using the fact that $|\partial A| \ge h |A|$ (since recall that $A$ cannot contain any vertex in the boundary ),
  \begin{equation}
      \P(A(f) = A) \le \frac{s}{t} \le (4M+1)^{|A|} \left(\frac{2M}{2M+1}\right)^{|\partial A|} \le e^{|A| \log (4M+1) - \frac{h|A|}{2M}}\label{eq:connected_component}
   \end{equation}
   Finally, the number of distance two connected sets $A$ with $|A|=n$ and $v_0 \in A$ is at most $d^{4n}$. Union bounding, we get
   $$
   \P(|A(f)| = n) \le  e^{n(4 \log d +  \log (4M+1) - \frac{h}{2M})} \le \exp\left(- \frac{hn}{4M}\right),
   $$
   as desired.

    It remains to prove that $T$ satisfies the required upper and lower bounds with~\eqref{eq:bound_s} and~\eqref{eq:bound_t}.
    We begin with the lower bound. Let $g \in \cE$.
    Since $|x-y| \le M$ for $x \in \{a,\dots,b\}$ and $y \in \{b-M,\dots,a+M\}$, and since $g(v) \in \{a,\dots,b\}$ for any even $v \in \partial(A \cup \partial A)$ and $g(v) \in \{b-M,\dots,a+M\}$ for any odd $v \in \partial(A \cup \partial A)$, we see that every function satisfying~\eqref{eq:f'} must necessarily belong to $\cL$. Hence,
  $$
  |T(g)| = (b-a+1)^{|\{v \in  A \cup \partial A: v \text{ odd}|}(a-b+2M+1)^{|\{v \in  A \cup \partial A: v \text{ even}|} = t.
  $$
    We now move on to the upper bound. Fix $g' \in \cL$. We need to upper bound the number of $g \in \cE$ such that $g' \in T(g)$. Clearly, $g$ is uniquely determined outside $A \cup \partial A$. Note that if a vertex $u$ in $A$ is at distance 2 from the complement of $A \cup \partial A$, then the number of possible values of $g(u)$ is at most $4M+1$. By choosing the values of $g$ on $A$ one after another, in any order consistent with their distance from the complement of $A \cup \partial A$, we see that there are at most $(4M+1)^{|A|}$ possible choices for $g$ on $A$. To show that $|\{ g\in \cE : g' \in T(g)\}| \le s$, it remains to show that, given $g$ outside $\partial A$, the number of choices for $g$ on $\partial A$ is at most $\prod_{v \in \partial A} \ell_v$, where
    $$
l_v:= \begin{cases} b-a &\text{if $v$ is even }\\ 2M+a-b &\text{if $v$ is odd} \end{cases} .
$$
    For this, it suffices to prove a vertex-by-vertex bound, namely, that the number of choices for $g(v)$ is at most $\ell_v$ for any $v \in \partial A$. To this end, fix $v \in \partial A$, and let $u \in A$ be any neighbor of $v$ (which exists since $v \in \partial A$).
    Suppose first that $v$ is even.    
    First note that since $v \in \partial A$ and $v$ is even, $g(v) \in \{a,\ldots,b\}$.
    Similarly, since $u \in A$ and $u$ is odd, $g(u) \notin \{b-M,\dots,a+M\}$.
    Consequently,
    \begin{align*}
    a+1 \le g(u) - M  \le g(v) \le b, \qquad\text{if $g(u) \ge a+M+1$},\\
    a \le g(v) \le g(u)+M \le b-1, \qquad\text{if $g(u) \le b-M-1$.}
    \end{align*}
    (One consequence of the above inequalities is that $v$ cannot be even if $a=b$.)
    Therefore, the number of possible values of $g(v)$ is at most $\ell_v$, as claimed.
    Suppose next that $v$ is odd.
    We similarly have that $g(v) \in \{b-M,\dots,a+M\}$ and $g(u) \notin \{a,\dots,b\}$, so that
    \begin{align*}
        b-M+1 \le g(u) - M \le g(v) \le a+M \qquad\text{if $g(u) \ge b+1$, }\\
         b-M \le g(v) \le g(u)+M \le a+M-1 \qquad\text{if $g(u) \le a-1$.}
    \end{align*}
    (Again, $v$ cannot be odd if $b-a=2M$.)
    Therefore, the number of possible values of $g(v)$ is again at most $\ell_v$, as claimed.
   \end{proof}
   
The following corollary is a simple generalization of \eqref{eq:connected_component}, and is what we will actually need in the proof of \Cref{thm:main_general}.
   \begin{corollary}\label{cor:generalization_exp_tail}
       Let $A_1 ,A_2,\ldots, A_k$ be disjoint two connected sets of total size $n := \sum_{i=1}^k|A_i|$. Then the probability that $A_1,\ldots, A_k$ are distance two connected components of $B(f)$ is at most 
       $$
       \exp\left(n\big(\log(4M+1)  - \tfrac h{2M}\big)\right).
       $$
   \end{corollary}

\begin{proof}[Proof of \Cref{thm:main_general}]
    Let $(G_n)$ be an exhaustion.
    Note that it is enough to prove the following: for all $v \in V$, $R>0$ and for all connected $H,H'$ such that $B(v,R) \subset B(v,R') \subset H\subset H' \subset G$, $\nu_{H,M}^{a,b}$ can be coupled with $\nu_{H',M}^{a,b}$ so that they agree on $B(v, R)$ with probability tending to $1$ as $R' \to \infty$.  
    
    To construct the coupling, let $f \sim \nu_{H, M}^{a,b}$ and $f' \sim \nu_{H',M}^{a,b}$ be sampled independently and let $B(f)$ and $B(f')$ be the set of `atypical' vertices as defined before \Cref{prop:cluster_exp_tail}. Let $U = B(v,R)$ and let $A(f,f',U)$ be the union of the distance two connected components of $B(f) \cup B(f')$ intersecting $U$. Then for any $A$, using \Cref{cor:generalization_exp_tail},
    \begin{align*}
    \P(A(f,f',U) = A) &\le \sum_{A_1 \cup A_2 = A} \exp((|A_1| +|A_2|)(\log (4M+1) - \tfrac h{2M})) \\ &\le 3^{|A|} \exp(|A|(\log (4M+1) - \tfrac h{2M})) ).
    \end{align*}
    Since there are at most $(2k)^{|U|}d^{4k}$ possible sets of size $k$ containing $U$,
    $$
    \P(|A(f,f')| = k) \le \exp\left(k (\log (3(4M+1)) - \tfrac h{2M} + 4\log d) + |U|\log (2k)\right) \le \exp(-\tfrac{hk}{4M} + |U|\log(2k)).
    $$
    Since $U$ is fixed, this bound decays exponentially in $k$.
    In particular, as $R' \to \infty$,
    $$
    \P(A(f,f',U) \subset B(v,R')) \to 1.
    $$

    It remains to show that $f$ and $f'$ can be coupled so that they agree on $B(v,R)$ on the event $\{(A(f,f',U) \subset B(v,R')\}$. We proceed via exploration from the boundary.
    Let $X$ be the union of the distance two connected components of $B(f) \cup B(f') \cup (G \setminus H)$ containing $G \setminus H$. Note that, by definition of $A(f,f',U)$, we have that $X$ is disjoint from $B(v,R+2)$ whenever $A(f,f',U) \subset B(v,R'-2)$. Denote $\tilde X := X \cup \partial X$, $\bar X := \tilde X \cup \partial \tilde X$ and $Y := G \setminus \bar X$. Note that $X$ is measurable with respect to $(B(f) \cup B(f')) \cap \bar X$.
    Now condition on $X$ and on $(f|_{\tilde X},f'|_{\tilde X})$ and suppose that $X$ is disjoint from $B(v,R+2)$. We claim that $f|_Y$ and $f'|_Y$ are conditionally distributed according to $\nu^{a,b}_{Y,M}$ (in fact, they are also conditionally independent, but we do not need this), and in particular, $f|_Y$ and $f'|_Y$ can be coupled to agree on $Y$ (which contains $B(v,R)$). 
    We prove this for $f|_Y$ (the proof for $f'_Y$ is identical).
    Indeed, every even $v \in \partial Y$ has $f(v) \in \{a,\dots,b\}$, since $\partial Y \subset \partial \tilde X \subset (B(f) \cup B(f'))^c$. On the other hand, any value in $\{a,\dots,b\}$ is possible for $f(v)$, since any $u \in N(v) \setminus Y \subset \bar X \setminus X \subset (B(f) \cup B(f'))^c$, is odd, and satisfies $f(u) \in \{b-M,\dots,a+M\}$.
    Similarly, any odd vertex in $\partial Y$ must take a values in $\{b-M, \ldots, a+M\}$ and is free to take any value there. Since there are no other constraints (beyond the obvious Lipschitz constraint), the conditional law of $f|_Y$ is $\nu^{a,b}_{Y,M}$. This completes the proof of convergence.

Now we show part (b) of the theorem, which is an easy consequence of \Cref{prop:cluster_exp_tail}.
Let $f \sim \nu^{a,b}_{G,M}$. Fix an odd vertex $u$ and denote $\cN := \{ f(v) : v \in N(u) \}$. Let $\cN^-$ and $\cN^+$ be the minimum and maximum elements in $\cN$, respectively.
Since $f(u)$ is conditionally uniformly distributed in $\{\cN^+-M,\dots,\cN^-+M\}$ given $\cN$, we see that
\begin{align*}
 \P(\{\cN^-,\dots,\cN^+\} \subsetneq \{a,\dots,b\})
  &\le (2M+1) \cdot \P(f(u) \notin \{b-M,\dots,a+M\}) \\
  &\le (2M+1) \cdot \P(|A(f,u)|\ge1)) \\&\le \frac{2M+1}{e^{h/4M}-1} .
\end{align*}
On the other hand,
\begin{align*}
 \P(\{\cN^-,\dots,\cN^+\} \not\subset \{a,\dots,b\})
  &= \P(\cN^- < a\text{ or }\cN^+>b) \\
  &= \P(f(v) \notin \{a,\dots,b\}\text{ for some }v \in N(u)) \\
  &\le \sum_{v \in N(u)} \P(|A(f,v)|\ge1)) \\&\le \frac{d}{e^{h/4M}-1} .
\end{align*}
Together, we get that
\[ \P(\{\cN^-,\dots,\cN^+\} \neq \{a,\dots,b\}) \le \frac{d+2M+1}{e^{h/4M}-1} \le \frac{d+2M+1}{3d^4(4M+1)-1} < \frac12 ,\]
which means that $\{a,\dots,b\}$ is the largest atom of $\{\cN^-,\dots,\cN^+\}$, so that it is uniquely determined by $\nu^{a,b}_{G,M}$.
\end{proof}

\section{FKG}\label{sec:FKG}
In this section, we prove \Cref{thm:main_lipschitz_FKG}. 
We establish certain monotonicity condition for uniform Lipschitz functions and the absolute value of them. Such results have also been obtained independently by Karrila \cite{karrila2023logarithmic}.

Let $G = (V,E)$ be a finite graph and let $\partial$ be a subset of its vertices which we call the \emph{boundary} of $G$. We will work with a slightly more general boundary condition in this section, where the constraints can vary over the vertices of $\partial$. A \emph{boundary condition} is an assignment $\{\kappa_v\}_{v \in \partial} $ where $\kappa_v \subset \Z$. Let $\cL(G,\partial,\kappa)$ be the set of Lipschitz functions $f$ on $G$ having $f(v) \in \kappa_v$ for all $v \in \partial$.
We will only consider cases where $|\cL(G, \partial, \kappa)| \in (0,\infty)$, and we will call such boundary conditions \emph{admissible}. For an admissible $\kappa$, let $\nu_G^{\partial, \kappa}$ denote the law of a uniform element drawn from $\cL (G,\partial,\kappa)$. 

A function $f:\Z^V \mapsto \R$ is increasing if for any $h,h' $ such that $h_v \le h'_v$ for all $v \in V$, $f(h) \le f(h')$. We denote by $\nu(f)$ the expectation of $f$ under the law $\nu$.
\begin{prop} (FKG for $h$)\label{prop:FKGh}
Take a finite graph $G=(V,E)$ with boundary $\partial $. Suppose $\kappa, \kappa'$ are two admissible boundary conditions such that for every $v \in \partial$, $\kappa_v  = [a_v,b_v]$ and $\kappa'_v = [a'_v,b'_v]$ and $a_v \le a'_v$, $b_v \le b'_v$ (these integers could be $\pm \infty$). Then 
\begin{description}
\item[(CBC)] For every increasing function $f$, $\nu_G^{\partial, \kappa'} (f) \ge \nu_G^{\partial, \kappa}(f) $.
\item[(FKG)] For two increasing functions $f$, $g$,  $\nu_G^{\partial, \kappa} (fg) \ge \nu_G^{\partial, \kappa} (f) \nu_G^{\partial, \kappa} (g) $.
\end{description}
\end{prop}
\begin{proof}
Both claims follow from Holley's criterion (\cite[Theorem 2.3]{grimmett2006random}). Fix a vertex $v$ and condition on $\{\xi_v: v \in V \setminus \{v\}\}$ in $\cL (G,\partial,\kappa)$. Now note that $\max \{|\xi_u  - \xi_{v}|: u \sim v\}$ is in $\{0,1,2\}$. If it is $2$, then the conditional value at $v$ is deterministic. If it is $1$, then the height at $v$ takes two values with equal probability. Finally if it is $0$, then the height at $v$ takes three values with equal probability. From this observation, it is easy to see that conditioning on $\{\xi'_v: v \in V \setminus \{v\}\}$ with $\xi' \ge \xi$ stochastically increases the height at $h$.
\end{proof}

A useful tool is monotonicity of the pushforward of the law of $\cL (G,\partial,\kappa)$ under the mapping $h \mapsto |h|$. Unfortunately it is not true directly as can be seen in the segment graph with three vertices with two vertices having degree one and one having degree 2 (call it $v$). If the the height is $0$ on the neighbours of $v$, then the conditional law of the height at $v$ is Bernoulli $(2/3)$. On the other hand conditioning one of absolute value heights of the neighbours to be 1 makes the conditional law at $v$ Bernoulli $(1/2)$.

It turns out that a different form of absolute value FKG holds, if we do the mapping $h \mapsto |h+0.5|$. Let $\cL^{>0}(G,\partial,\kappa)$ denote the set of $1$-Lipschitz functions taking values in $\N -0.5:= \{0.5,1.5, \ldots \}$. We also need to be careful with the boundary conditions. We say an admissible boundary condition $\kappa$ is $|h|$-adapted \footnote{Although we shift by $0.5$, we keep this notation, which hopefully is not a source of confusion.} if we can break up $\partial $ into $\partial  = \partial _{\rm pos} \sqcup \partial_{\rm sym}$ such that 
\begin{itemize}
\item $\kappa_v = -\kappa_v$ for every $v \in \partial_{\rm sym}$,
\item $\kappa_v \subseteq \N+0.5 := \{1.5,2.5,\ldots\}$ for every $v \in \partial_{\rm pos}$.
\end{itemize}
For every $\xi \in \cL^{>0}(G,\partial,\kappa)$, connect two vertices $u,v$ by an edge if $\max\{\xi_u, \xi_v \} >0.5$. Let $k(\xi)$ denote the number of connected components which do not intersect $B_{\rm pos}$. Then clearly for any $|h|$-adapted boundary condition $\kappa$,
\begin{equation}
\nu_G^{\partial, \kappa} (|h+0.5| = \xi) = \frac{2^{k(\xi)}}{|\cL(G, \partial, \kappa)|}, \qquad \xi \in \cL^{>0}(G, \partial, \kappa).\label{eq:modh_law}
\end{equation}
\begin{prop}[monotonicity for $|h+0.5|$]\label{prop:FKG_modh}
    Take a finite graph $G=(V,E)$ with boundary $\partial $. Let $\kappa, \kappa'$ be two $|h|$-adapted boundary conditions with $\partial_{\rm pos}(\kappa) \subseteq \partial_{\rm pos}(\kappa')$ and for every $v\in \partial$, $[a_v,b_v]:=\kappa_v\cap\{0.5,1.5,\ldots\} $ and $[a'_v,b'_v]:=\kappa'_v\cap \{0.5,1.5,\ldots\}$ satisfy $a_v\le a'_v$ and $b_v\le b'_v$ (these numbers could be $\infty$),
\begin{description}
 \item[(CBC)] For every increasing function $f$, 
 $\nu^{\partial, \kappa'}_G[f(|h+0.5|)]\ge\nu^{\partial, \kappa}_G[f(|h+0.5|)]$;
 \item[(FKG)] For any two increasing functions $f,g$, $\nu^ {\partial, \kappa}_G[f(|h+0.5|)g(|h+0.5|)] \ge   \nu^ {\partial, \kappa}_G[f(|h+0.5|)]  \nu^ {\partial, \kappa}_G[g(|h+0.5|)]$.
 \end{description}
 \end{prop}
\begin{proof}
The proof is very similar to \cite[Proposition 2.2]{duminil2022logarithmic}. 
 
 Fix a vertex $v $. Using \cite[Theorem 4.8]{georgii2001random} and references therein, { it is sufficient to prove that for any $\xi$ (resp. $\eta$) which are restrictions to $V \setminus \{v\}$ of configurations  in $\cL^{> 0}(G,\partial,\kappa)$ (resp. $\cL^{> 0}(G,\partial,\kappa')$)  such that $\xi_v \le \eta_v$ for all $v$}, and every $k  \ge0$,
\begin{equation}
\nu_{G}^{\partial,\kappa}\big[|h_v+0.5| \ge k\big| |h_{|V\setminus\{v\}}+0.5|= \xi\big] \le \nu_{G}^{\partial,\kappa'}\big[|h_v+0.5| \ge k
 \big| |h+0.5|_{|V\setminus\{v\}}= \eta \big]\label{eq:Holley_cond}.
 \end{equation}
If $\xi_u  \ge 2.5$ for some neighbour $u$ of $v$, or if $\xi_u = 1.5$ for all neighbours $u$ of $v$, then it is straightforward to see using \eqref{eq:modh_law} that \eqref{eq:Holley_cond} holds, as in this case we are simply reduced to the analysis in \cref{prop:FKGh} (the quantity $k(\xi)$ in \eqref{eq:modh_law} does not change whatever be the value of $|h_v+0.5|$, consequently $|h_v+0.5|$ is equally likely among all possibilities dictated by the Lipchitz constraint or the constraint imposed by $\kappa, \kappa'$ in case $v$ is a boundary vertex). 

Now suppose $\xi_u \le  1.5$ for every neighbour of $v$, and $\xi=0.5$ for at least one of the neighbours of $v$ (whence $|h_v+0.5| \in \{0.5,1.5\}$). Notice that all the neighbours $u$ with $\xi_u = 1.5$ must belong to the same component in $\xi$, call it $N_v(\xi)$. {If $v \in \partial$ and $\{0.5,1.5\} \not \subseteq \kappa_v$ then the value of $|h+0.5| $ at $v$ is deterministically equal to $\kappa_v \cap \{0.5,1.5\}$.} Otherwise, let $k'_v(\xi)$ denote the number of components of $k(\xi)$ containing at least one of the neighbours $u$ of $v$ with $\xi_u =0.5$ and which do not intersect $\partial_{\rm pos}(\kappa)$. If $N_v(\xi)$ does not intersect $\partial_{\rm pos}(\kappa)$, let $k_v(\xi) = k'_v(\xi)+1$, otherwise $k_v(\xi) = k'_v(\xi)$. Let $\cB(\kappa)$ denote the event that at least one of the components of $k(\xi)$ containing at least one of the neighbours of $v$ intersects $\partial_{\rm pos}(\kappa)$.
\begin{equation*}
\nu_{G}^{\partial,\kappa}\big[|h_v+0.5| =0.5\big| |h+0.5|_{|V\setminus\{v\}}= \xi\big] =
\begin{cases}
0 \text{ if }v \in \partial_{\rm pos}\\
 2^{k_v(\xi)} /(2 + 2^{k_v(\xi)}) \text{ if  $\xi \notin \cB(\kappa)$, $v \notin \partial_{\rm pos}$}\\
 2^{k_v(\xi)} /(1 + 2^{k_v(\xi)})\text{ if  $\xi \in \cB(\kappa)$, $v \notin \partial_{\rm pos}$}
\end{cases}
\end{equation*}
The same formula holds for $\eta$ as well. Now observe that $k_v(\xi) \ge k_v(\eta)$ since $\eta _v \ge \xi_v$ for all $v$ and $\partial_{\rm pos}(\kappa) \subseteq \partial_{\rm pos}(\kappa')$. If there is strict inequality $k_v(\xi) > k_v(\eta)$ then the above expression for $\xi$ is at least that for $\eta$ in all the cases by monotonicity of $x \mapsto x/(1+x)$ and $x\mapsto x/(2+x)$ and the fact that $2^{k}/(2+2^k)\ge 2^{k'}/(1+2^{k'}) $ if $k\ge k'+1$. So assume $k_v(\xi) = k_v(\eta)$. If $\xi, \eta$ are both in $\cB(\kappa)$ or both not in $\cB(\kappa)$, we are also done for similar reasons. 
The only case remaining is if $\eta \in \cB(\kappa')$ but $\xi \notin \cB(\kappa)$ (recall $\cB(\kappa) \subseteq \cB(\kappa')$) and $k_v(\xi) = k_v(\eta)$. We claim that this is impossible, as in this case  the strict inequality $k_v(\xi) > k_v(\eta)$ must hold. Indeed, in this case there must be a neighbour of $v$ whose cluster intersects $B_{\rm pos}(\kappa')$ and is not counted in $k_v(\eta)$, but every neighbour cluster is counted in $\xi$ as none of them intersect $B_{\rm pos}(\kappa)$ on the event $\cB(\kappa)$.\ 
\end{proof}
\begin{corollary}\label{cor:dichotomy}
Let $G$ be an infinite graph with a fixed vertex $\rho$. Let $(G_n)_{n \ge 1}$ be a sequence such that $G_n \subseteq G_{n+1}$ and $\cup G_n = G$. Suppose $\partial_n$ be a subset of vertices of $G_n$ such that $\partial_{n+1} \cap G_n \subseteq \partial_n$, and the distance between $\partial_n$ and $\rho$ converges to $\infty$.  Let $\kappa^n_v = \{-0.5,0.5\}$ for all $v \in \partial_n$. Then one of the following two cases hold.
\begin{description}
\item[(Localization)] $\nu^{G_n}_{\partial_n, \kappa^n}$ converges to a translation-invariant Gibbs measure supported on Lipschitz functions on $G$.
\item [(Delocalization)]$\lim_{m \to \infty} \sup_{n  }\mu^{G_n}_{\partial_n, \kappa^n}(|h(\rho)| > m) >0$.
\end{description}

\end{corollary}
\begin{proof}
Let $h^n \sim \mu^{G_n}_{\partial_n, \kappa^n}$.
Observe that $|h^n+0.5|$ is stochastically increasing in $n$, and hence if the delocalization condition does not hold, $\{|h^n+0.5|\}_{n \ge1}$ is tight. Consequently, weak limit of $|h^n+0.5|$ exists as a probability measure on $\Z^V$. Furthermore, observe that for each $n$, the height function $h^n+0.5$  is obtained by assigning a uniformly random sign to each connected component of $\xi$ as described in \Cref{eq:modh_law}. Thus this rule to obtain the height function from the absolute value also persists in the weak limit, and the proof of the corollary is complete.
\end{proof}

\begin{proof}[Proof of \Cref{thm:main_lipschitz_FKG}]
    Using \Cref{prop:FKGh}, $\mu_{n,d}^{\{0,1\}}$ is dominated by $\mu_{n,d}^{1}$ and dominates $\mu_{n,d}^{0}$, both of which are tight by \Cref{thm:main_lipschitz}. Thus $\mu_{n,d}^{\{0,1\}}$ is tight. Thus $\mu_{n,d}^{\{0,1\}}$ is in the localization phase of \Cref{cor:dichotomy}, and the result follows.
\end{proof}

\section{Discussion and open questions}\label{sec:open}
In this section, we briefly discuss some properties of the limiting measures $\mu^{S}_{d}$ for finite subsets $S \subset \Z$. Let us first show how to construct translation invariant Gibbs measures on $d$-regular trees rather than $d$-ary trees, the nice thing about the former is that they are vertex transitive. The convergence result for the $d$-regular tree actually follows from that of the $d$-ary tree in a straightforward manner: Let $\tilde f_n \sim \mu^S_{d,n}$ and let $\tilde f_n \sim \tilde \mu^S_{d,n}$ denote the analogous law for the $d+1$-regular tree. Then 
\begin{equation}
    \P(\tilde f_n(\rho) = k) = F_{d+1}(F_d^{(n-1)} (\delta_S))_k \nonumber
\end{equation}
where $F_d$ is the map $F$ for $d$-ary tree as described in \Cref{sec:apriori_bounds}.  Thus the limit as $n \to \infty$ exists for the $d+1$-regular tree exists whenever it exists for the $d$-ary tree tree.
Furthermore if $\pi^S_d:=\pi^S_d(\rho)$ denote the law of the height at $\rho$ under $\mu^{S}_d$, whenever it exists, then the corresponding law for the $d+1$-regular tree is given by $F_{d+1}(\pi_d^S)$.
It is also easy to check that the transition matrix for the tree indexed Markov chain  for the $d$-ary tree as well as the $d+1$ regular tree is given by 
$$
(P^S_d)_{i,j} = \frac{\pi_d^S(j)}{\pi_d^S(i-1)+\pi_d^S(i)+\pi_d^S(i+1)}1_{|j-i|\le 1}.
$$
and that $F_{d+1}(\pi^S_d)$ is the stationary measure for this transition matrix. Using reversibility of the chain, translation invariance of the limit for the $d+1$-regular tree is immediate.

\bigskip
We now speculate about the most general result possible for general boundary conditions in our setup.
We saw that for small $d$ convergence holds for all the cases we considered here. For large $d$ however, such convergence holds along the even sequences.
This leads us to conjecture:
\begin{conjecture}
    $\lim_{n \to \infty }\mu_{n,2}^S$ exists for all finite $S$. Also, $\lim_{n \to \infty }\mu_{2n,d}^S$ exists for all finite $S$, and for all $d \ge 2$.
\end{conjecture}
We believe that for $S = \{-k,\ldots, k\}$ and $d \ge 8$,  a more careful analysis similar to that in \Cref{sec:phi_estimate} should be enough to conclude the convergence of the map $\psi \circ \psi$.

Next comes the question of understanding all possible translation invariant measures we can get as limits of the above type.
To summarize the results of this article, we constructed the translation invariant Gibbs measures $\tilde \mu_d^{S}$, $S = \{-k,\ldots, k\}$ for $2 \le d \le 7$ and for $S=\{0,1\}$, $d \ge 2$. We also proved that for $2 \le d \le 7$, $\tilde \mu_d^{\{0,1\}}  \neq  \tilde \mu_d^0 = \tilde \mu_d^{\{-k,\ldots, k\}}$ for all $k \ge 0$. This raised the following research program.
\begin{question}
    Characterize all translation invariant measures of the form $\tilde \mu_d^S$ where $S \subset \Z$ is finite, whenever they exist. In particular, is there a natural set of properties such that if $S_1,S_2$ satisfies all of the properties, then it is guaranteed that $\tilde \mu_d^{S_1} = \tilde \mu_d^{S_2}$?
\end{question}

Another interesting direction is to consider $M$-lipschitz functions for $M>1$.
\begin{question}
    Extend \cref{thm:main_lipschitz_FKG} to $M>1$: Does $\mu^{\{0,\dots,M\}}_{n,d,M}$ converges for all $d \ge 2$ and $M \ge 2$?
\end{question}
\begin{question}
    Extend \cref{thm:main_general} to $M>1$: Given $M>1$, show that $\mu^0_{n,d,M}$ converges if and only if $d \le d_c(M)$. Can one at least show convergence for $M \gg d$?
\end{question}

Another very interesting avenue of research is extremality. The measures $\tilde \mu_d^{0}$ and $\tilde \mu_d^{\{0,1\}}$ are both mixing, hence ergodic. However, the question of whether they are extremal remains:
\begin{question}
    For what values of $d$ are the measures $\mu_d^{\{0,1\}}$ and $\mu_{d}^0$ extremal for $2 \le d \le 7$?
\end{question}
This is related to the question of extremality of the Ising model, see \cite[Section 3]{lyons_factor_tree}.

\section{Estimates of fixed point of $\varphi$}\label{sec:phi_estimate}
In this section we complete the proof of \Cref{prop:estimate_fixed_point}. Recall from \Cref{lem:S-abc-iteration} that  
\begin{equation*}
    \psi(\cS_{a,b,c}) \subset \cS_{a',b',c'},
\end{equation*}
where 
\[(a',b',c')= \varphi(a,b,c) := (f(1,b), f(1+c,a), g(b,c))  ,\]
and
\begin{align*}
    f(\alpha,x) &:= \left(\frac{1+\alpha x}{1+2x}\right)^d, \\
    g(b,x) &:= b^d \left(\frac{1+x+x^2}{1+b+bx}\right)^d. 
\end{align*}

Let us summarize the basic strategy to approximate the fixed point of $\varphi$.
\begin{itemize}
\item First we prove in \Cref{lem:fixed_point_of_upper_bound} that $x \mapsto g(b,x)$ has a unique fixed point in $(0,1)$.
    \item Since $\cS_{\varphi(0,1,c)} \subset \cS_{0,1,g(1,c)}$, we can iterate this bound many times to reach close to $\cS_{0,1,c_g}$ where $c_g$ is the unique fixed point of $x\mapsto g(1,x)$ in $(0,1)$.

    \item Next we note that $$\psi \circ \psi(\cS_{a,b,c_g}))  \subset  \cS_{{\sf i}(c_g,a), {\sf j}(c_g,a), c_g}$$
    where ${\sf i}, {\sf j}$ are defined as in \eqref{eq:i,j}. The reason for this approximation is that now it is enough to analyze the one dimensional maps ${\sf i}, {\sf j}$ (with $c_g$ fixed).
    \item We prove in \Cref{lem:two_fold_iteration} next that ${\sf i}, {\sf j}$ have derivatives  in $(0,1)$ which (along with some other basic properties) is enough to prove that each of them has a unique fixed point in $(0,1)$. We call them $a_{c_g}, b_{c_g}$. In conclusion, after enough iterates we approximately reach $\cS_{a_{c_g}, b_{c_g}, c_g}$.
    
    \item It turns out that $a_{c_g}, b_{c_g}, c_g$ can be taken  to be our estimates $\hat a_*, \hat b_*, \hat c_*$ for $3 \le d \le 6$. Unfortunately for $d=2,7$, these estimates are not enough. For these two cases we need another `round' of approximation as follows.

    \item In \Cref{lem:c_2nditerate} we prove using similar ideas that starting close to $\cS_{a_{c_g}, b_{c_g}, c_g}$, we can reach $\cS_{a_{c_g}, b_{c_g}, \tilde c_g}$ where $\tilde c_g$ is the unique fixed point of $x\mapsto g(b_g',x)$ where $b_g'=\max(f(1+c_g,a_{c_g}), b_g)$.
    \item Finall in \Cref{lem:iteration_ab_2nd}, we prove that starting close to $\cS_{a_{c_g}, b_{c_g}, \tilde c_g}$ we reach $\cS_{a_{\tilde c_g}, b_{\tilde c_g}, \tilde c_g}$ where $a_{\tilde c_g}, b_{\tilde c_g}$ are unique fixed points of $x \mapsto {\sf i}(x,\tilde c_g), x\mapsto {\sf j}(x, \tilde c_g)$.
    \item Now we need to explain how we get hold of the values close enough to the fixed points $c_g, a_{c_g}, b_{c_g}, \tilde c_g, a_{\tilde c_g}, b_{\tilde c_g}$. We prove in \Cref{lem:fixed_g} that the map $x >g(b,x)$  if and only if $x <c_g$ where $c_g$ is the fixed point of $x \mapsto g(b,x)$. Thus if we can produce two points $c_1, c_2$ such that $g(b,c_1) > c_1$ and $g(b,c_2)<c_2$ then $c_1<c_g<c_2$.
    In practice, we can `guess' good estimates of $c_1,c_2$ by looking at the plots, and then plug them into $g(b,x)$ to confirm that indeed $g(b,c_1)>c_1$, $g(b,c_2)<c_2$ which makes the estimates completely rigorous. The maps $x \mapsto {\sf i}(c,x), x \mapsto {\sf j}(c,x)$ also satisfy similar properties which we exploit in a similar manner to get good estimates of $a_{c_g}, b_{c_g}$ and then $a_{\tilde c_g}, b_{\tilde c_g}$. This part is done in a separate subsection \Cref{sec:anal_approx}.
\end{itemize}

We now explain the details of the above strategy.

\begin{lemma}\label{lem:fixed_point_of_upper_bound}
For $x>0,b\in (0,1]$, $x \mapsto g(b,x)$ is increasing, convex and has  a unique fixed point $c_0$  in $(0,1)$. Moreover, if $0 \le c<1$ then for any fixed $b$, defining $g_b(c) = g(b,c)$,  $g_b^{(i)}(c)$ converges to $c_0 = c_0(b)$. Also, $g_b(x)>x$ if and only if $x<c_0(b)$.
\end{lemma}
\begin{proof}
We have 
$$
\frac{\partial g}{\partial x} = \frac{db^d(bx^2+2bx+2x+1)(x^2+x+1)^{d-1}}{(bx+b+1)^{d+1}}
$$
and 
\begin{multline*}
\frac{\partial^2 g}{\partial x^2} = d b^d (x^2 + x + 1)^{d - 2} (b x + b + 1)^{-d - 2}\\
(b^2 ((d - 1) x^4 + 4 (d - 1) x^3 + (4 d - 2) x^2 + 2 x + 2) + 2 b (2 (d - 1) x^3 + (5 d - 4) x^2 + (2 d - 1) x + 1) + d (2 x + 1)^2 - 2 x^2 - 2 x + 1)
\end{multline*}
It can be easily checked from these expressions that $\partial g/\partial x$ and $\partial^2g/\partial x^2$ are both strictly positive for $d \ge 2$. 

For $b<1$, since $g(b,0) = b^d >0$ and $g(b,1)  < g(1,1) = 1$ and $g $ is strictly increasing and convex, it must be the case that $g$ has a unique fixed point in $(0,1)$. If $b=1$, although $g(1)=1$ and we further note that $$\frac{\partial g}{\partial x}|_{x=1} = \frac{2d}{3}>1$$
if $d \ge 2$. Thus $g$ has a unique fixed point in $(0,1)$ in this case also.  This takes care of the `also' part in the lemma as well.

The fact that $g^{(i)}(c)$ converges to the fixed point $c_0$ is immediate from the fact that $c_0<g(x)<x$ if $x>c_0$ and $c_0>g(x)>x$ if $x<c_0$, and hence the iterates monotonically converge to $c_0$.
\end{proof}

We now introduce a convenient notion of convergence of sets of the form $\cS_{a,b,c}$. We say $\cS_{a,b,c} \xrightarrow[]{\psi} \cS_{a_0,b_0,c_0}$ if for all $x \in \cS_{a,b,c}$ and all $\ve>0$, there exists a $K$ such that for all $k \ge K$, $\psi^{(k)}(x) \in \cS_{a_0-\eps,b_0+\eps,c_0+\eps}$. Similarly, we define $\cS_{a,b,c} \xrightarrow[]{\psi^2} \cS_{a_0,b_0,c_0}$ for all $x \in \cS_{a,b,c}$ and all $\ve>0$, there exists a $K$ such that for all $k \ge K$, $\psi^{(2k)}(x) \in \cS_{a_0-\eps,b_0+\eps,c_0+\eps}$.

\begin{lemma}\label{lem:fixed_g}
Fix $c<1$.
    We have $$\cS_{0,1,c} \xrightarrow[]{\psi} \cS_{0,1,c_g},$$ 
    where $c_g$ is the unique fixed point of $g(1,\cdot)$ in $(0,1)$ (as guaranteed by \Cref{lem:fixed_point_of_upper_bound}).
\end{lemma}
\begin{proof}
    This is immediate from \Cref{lem:S-abc-iteration,lem:fixed_point_of_upper_bound}.
\end{proof}

Now we plan to bootstrap this estimate into an iteration of the first two coordinates.
To that end, define 
\begin{equation}
   {\sf i}(c,x) := f(1,f(1+c,x)), \qquad {\sf j}(c,x) := f(1+c, f(1,x)). \label{eq:i,j}
\end{equation}

{To motivate these definitions, note that ${\sf i}(c,a)$ is precisely the first coordinate of $\varphi^{(2)}(a,b,c)$. It turns out that ${\sf j}(c,b)$ serves as a bound on the second coordinate of $\varphi^{(2)}(a,b,c)$ when $c=c_g$.}

\begin{lemma} \label{lem:two_fold_iteration}
    For all $2 \le d \le 7$, there exists a $0<c_g<c_d $  such that for all $c<c_d$, both $x \mapsto {\sf i}(c,x)$ and $x \mapsto {\sf j}(c,x)$ have unique fixed points $a_{ c},b_{c}$ in $(0,1)$. Furthermore, both $|\partial {\sf i}/\partial x |<1$ and $|\partial {\sf j}/\partial x |<1$ in $[0,1]$. Consequently for any $x \in [0,1]$, $${\sf i}^{(k)}(x)  \to a_{c} \text{ and }{\sf j}^{(k)}(x)  \to b_{c}\text{ as } k \to \infty.$$
    Furthermore, ${\sf i}(c,x)>x$ if and only if $x \in [0,a_c)$ and ${\sf j}(c,x)>x$ if and only if $x \in [0,b_c)$.
\end{lemma}
We postpone the proof of \Cref{lem:two_fold_iteration} to later in order to maintain the flow of the argument.
\begin{lemma}\label{lem:iterate_ab}
    Let $c_g$ be as in \Cref{lem:fixed_g} and $a_{c_g},b_{c_g}$ as in \cref{lem:two_fold_iteration} for $c=c_g$.  Then for all $c<1$
    $$
    \cS_{0,1,c} \xrightarrow[]{\psi^2} \cS_{a_{c_g}, b_{c_g}, c_g}.
    $$
    as $k \to \infty$.
\end{lemma}
\begin{proof}
Applying \Cref{lem:fixed_g}, we see that for all $\delta\ge 0$ so that $c_g+\delta<1$ for all large enough $k$
\begin{equation*}
   \psi^{(k)} (\cS_{0,1,c} ) \subset \cS_{0,1,c_g+\delta} 
\end{equation*}
    Recall
\begin{equation*}
    \varphi (a,b,c_g+\delta) = (f(1,b), f(1+c_g+\delta,a), g(b,c_g+\delta))
\end{equation*}
Since $g(b,x)$ is increasing in $b$, and $g(1,x)\le x$ if and only if $x \ge c_g$, $g(b,c_g+\delta) \le g(1,c_g+\delta)  \le c_g+\delta$. 
Thus we have 
$$
\cS_{\varphi(a,b,c_g+\delta)} \subseteq \cS_{f(1,b), f(1+c_g+\delta,a), c_g+\delta}
$$
Now note 
$$
\varphi(f(1,b), f(1+c_g+\delta,a), c_g+\delta) = (f(1,f(1+c_g+\delta,a)), f(1+c_g+\delta, f(1,b)), g(f(1+c_g,a), c_g+\delta)).
$$
Note again that $g(f(1+c_g,a), c_g+\delta) \le g(1,c_g+\delta) \le c_g+\delta$. 
Thus we get
\begin{equation*}
  \cS_{\varphi(f(1,b), f(1+c_g+\delta,a), c_g+\delta)} \subset \cS_{{\sf i}(c_g+\delta,a),{\sf j}(c_g+\delta,b), c_g+\delta}.  
\end{equation*}
Overall 
\begin{equation*}
    \psi \circ \psi(\cS_{a,b,c_g+\delta}) \subset \psi (\cS_{\varphi(a,b,c_g+\delta)}) \subset \psi(\cS_{f(1,b), f(1+c_g+\delta,a), c_g}) \subset \cS_{\varphi(f(1,b), f(1+c_g+\delta,a), c_g+\delta)} \subset \cS_{{\sf i}(c_g+\delta,a),{\sf j}(c_g+\delta,b), c_g+\delta}.
\end{equation*}


Thus iterating, and using \Cref{lem:two_fold_iteration}, eventually the sequence enters $\cS_{a_{c_g(\delta)}-\eps, b_{c_g(\delta)}+\eps, c_g+\delta}$ for arbitrarily small $\eps$, where $a_{c_g(\delta)} ,b_{c_g(\delta)}$ are the unique fixed points of $x \mapsto {\sf i}(c_g+\delta, x)$  and ${\sf j}(c_g+\delta, x)$ in $(0,1)$ respectively. By continuity, $a_{c_g(\delta)} \to a_{c_g}, b_{c_g(\delta)} \to b_{c_g}$ as $\delta \to 0$. This allows us to conclude the lemma.
\end{proof}



Although the approximation $(a_{c_g},b_{c_g}, c_g)$ for the fixed point of $\varphi(a,b,c)$ is good enough for $3 \le d \le 6$, we need another round of iteration to get closer to the fixed point of $\varphi(a,b,c)$ for $d \in \{2,7\}$. To that end, define $$b_g' = \max(f(1+c_g,a_{c_g}), b_{c_g}).$$ Let $\tilde c_g$ be the fixed point of the map $x \mapsto g(b_g', x)$ (which is guaranteed to exist by \Cref{lem:fixed_g}).

\begin{lemma}\label{lem:c_2nditerate}
We have  for all $c<1$
\begin{equation*}
        \cS_{0, 1, c} \xrightarrow[]{\psi^2} \cS_{a_{c_g}, b_{c_g}, \tilde c_g}
\end{equation*}
as $k \to \infty$. Furthermore, $\tilde c_g < c_g<c_d$.
\end{lemma}
\begin{proof}
Since $b_g' <1$, $\tilde c_g = g(b_g',\tilde c_g) < g(1, \tilde c_g)$ which implies $\tilde c_g<c_g$.

Next, by \Cref{lem:iterate_ab}, for all $\delta$, 
$$
\psi^{(2k)} (\cS_{0,1,c}) \subset \cS_{a_{c_g} - \delta, b_{c_g} + \delta, c_g+\delta}.
$$
for all large enough $k$.
Pick $\delta$ small enough so that $c_g+\delta<c_d$ and $b_g''<1$ where
\begin{equation}
    b_g'' = \max (f(1+c_g+\delta,a_{c_g } -\delta), b_{c_g} + \delta). \label{eq:bg''}
\end{equation}
Note that
$$
g(b_g'',c_g ) < g(1,c_g) = c_g
$$
Hence by continuity we can pick a further small $\delta$ if required to conclude that $g(b_g'', c_g+\delta)  \le c_g$. 

We first show that after two more steps of iteration, we enter $\cS_{a_{c_g} - \delta, b_{c_g} + \delta, c_g}$. Indeed after one iteration
$$
\psi(\cS_{a_{c_g} - \delta, b_{c_g} + \delta, c_g}) \subset \cS_{f(1,b_{c_g}+\delta),f(1+c_g+\delta, a_{c_g} - \delta) ,g(b_{c_g-\delta},c_g)}
$$
by the choice of $k$. Now note since $b \mapsto g(b,x)$ is increasing in $b$, 
$$
g(b_{c_g-\delta}, c_g+\delta) \le g(b_g'', c_g+\delta) \le c_g 
$$
by the choice of $\delta$. By the same logic and the choice of $k$, after another iteration, we are in $\cS_{a_{c_g} - \delta, b_{c_g} + \delta, c_g}$.

We will show that for all $c\le c_g$, 
\begin{equation}
    \psi^{(2)}(\cS_{a_{c_g}-\delta,b_{c_g}+\delta,c}) \subseteq \cS_{a_{c_g}-\delta, b_{c_g}+\delta , g(b_g'',g(b_g'',c))} \text{ and } g(b_g'',g(b_g'',c)) \le c_g.
\end{equation}
Indeed by \Cref{lem:fixed_g}, this is enough since the fixed point of $x \mapsto g(b_g'',x)$ tends to that of $x \mapsto g(b_g',x)$ as $\delta \to 0$ by continuity.

Since  $b \mapsto g(b,x)$ is increasing (see \eqref{eq:g}) and $c \mapsto f(1+c,x)$ is increasing,
\begin{equation*}
    g(b_{c_g}+\delta, c) \le g(b_g'',c), \qquad \qquad f(1+c, a_{c_g}-\delta) \le f(1+c_g, a_{c_g}-\delta).
\end{equation*}
Thus we have 
\begin{equation*}
    \cS_{\varphi(a_{c_g}-\delta, b_{c_g}+\delta,c)} \subseteq \cS_{f(1,b_{c_g}+\delta), f(1+c_g,a_{c_g}-\delta), g(b_g'',c)}.
\end{equation*}
Now 
\begin{multline*}
    \varphi(f(1,b_{c_g}+\delta), f(1+c_g,a_{c_g}-\delta), g(b_g'',c)) \\= (f(1,f(1+c_g,a_{c_g}-\delta)), f(1+g(b_g'',c), f(1,b_{c_g}+\delta)), g(f(1+c_g,a_{c_g}-\delta), g(b_g'',c)) )
\end{multline*}
Now note,
\begin{equation*}
 f(1, f(1+c_g, a_{c_g} - \delta))  = {\sf i}(c_g,a_{c_g} - \delta) \ge a_{c_g}- \delta  
\end{equation*}
since ${\sf i}(c_g, x)>x$ if and only if $x<a_{c_g}$.
 Also since $b\mapsto g(b,c)$ and $c \mapsto g(b,c)$ are both increasing, and $g(1,x)<x$ if and only if $x>c_g$,
\begin{equation*}
    g(b_g'',c) \le g(b_g'', c_g) \le  c_g, \qquad g(f(1+c_g,a_{c_g}-\delta), g(b_g'',c))  \le g(b_g'', g(b_g'', c)).
\end{equation*}
and hence since $c \mapsto f(1+c,x)$ is increasing as well,
\begin{equation*}
    f(1+g(b_g'',c), f(1,b_{c_g}+\delta)) \le f(1+c_g, f(1, b_{c_g}+\delta))  \le  b_{c_g}+\delta.
\end{equation*}
where for the last equality we used the face that ${\sf }j(c_g,x)<x$ if and only if $x>b_{c_g}$.
Combining, we get
\begin{equation*}
    \cS_{\varphi(f(1,b_{c_g}+\delta), f(1+c_g,a_{c_g}-\delta), g(b_g'',c))} \subseteq \cS_{a_{c_g}-\delta, b_{c_g}+\delta, g(b_g'',g(b_g'',c))}.
\end{equation*}
We now need to show that $g(b_g'',g(b_g'',c)) \le c_g$, which is clear since 
$$
g(b_g'', g(b_g'',c)) \le g(b_g'', c_g) \le c_g
$$
by our choice of $\delta$.
The proof is complete.
\end{proof}

Next we bootstrap this to improve upon $a_{c_g}, b_{c_g}$. Let $a_{\tilde c_g}, b_{\tilde c_g}$ be the fixed points of $x \mapsto {\sf i}(\tilde c_g, x)$ and $ x \mapsto {\sf j}(\tilde c_g, x)$ respectively. Then 
\begin{lemma}\label{lem:iteration_ab_2nd}
We have for all $c<1$
\begin{equation*}
   \cS_{0, 1, c} \xrightarrow[]{\psi^2} \cS_{a_{\tilde c_g}, b_{\tilde c_g}, \tilde c_g}.
\end{equation*}
  as $k \to \infty$.  
\end{lemma}
\begin{proof}
First fix $\delta'>0$ so that $\tilde c_g + \delta' < c_g$. Then choose $\delta$ small enough such that letting $b_g''$ be as in \Cref{eq:bg''}, $\tilde c_g + \delta' $ is bigger than the fixed point of $x \mapsto g(b_g'',x)$.

By \Cref{lem:c_2nditerate}, we see that for all $k$ large enough
$$
\psi^{(2k)} (\cS_{0,1,c})\subset \cS_{a_{c_g}-\delta, b_{c_g+\delta}, \tilde c_g + \delta'}
$$
By continuity of $\varphi$, it is enough to show that  for all $a \ge a_{c_g} - \delta, b \le b_{c_g}+\delta$,
\begin{equation*}
\psi^{(2)}(\cS_{a,b,\tilde c_g+\delta'}) \subseteq \cS_{{\sf i}(\tilde c_g+\delta, a), {\sf j} (\tilde c_g+\delta, b), \tilde c_g+\delta'}, \qquad {\sf i}(\tilde c_g, a) \ge a_{c_g}-\delta, \quad {\sf j} (\tilde c_g, b) \le b_{c_g}+\delta.
\end{equation*}
    Note that 
    \begin{equation*}
        \varphi(a,b, \tilde c_g+\delta') = (f(1,b), f(1+\tilde c_g+\delta', a), g(b,\tilde c_g+\delta'))
    \end{equation*}
    and
    $$
    g(b,\tilde c_g+\delta') \le g(b_{c_g}+\delta, \tilde c_g+\delta') \le g(b_g'', \tilde c_g + \delta') \le \tilde c_g+\delta'.
    $$
    since $x > g(b_g'',x) $ if and only if $x $ is at least the fixed point of $x \mapsto g(b_g'',x)$ and $\delta' $ is chosen to satisfy this.
    Thus 
    \begin{equation*}
        \psi(\cS_{a,b, \tilde c_g} )\subseteq \psi(\cS_{f(1,b), f(1+\tilde c_g+\delta', a), \tilde c_g+\delta'}).
    \end{equation*}
    Again observe,
    $$
    \varphi(f(1,b), f(1+\tilde c_g+\delta', a), \tilde c_g+\delta') = ({\sf i}(\tilde c_g+\delta',a) , {\sf j}(\tilde c_g+\delta',b), g(f(1+\tilde c_g+\delta', a), \tilde c_g+\delta'))
    $$
    
    Now observe that since $a \ge a_{c_g} -\delta$ and $x \mapsto f(1+c,x)$ is decreasing, $c \mapsto f(1+c,x)$ is increasing,  and $\tilde c_g+\delta' < c_g$
    $$
    f(1+\tilde c_g+\delta', a) \le f(1+c_g+\delta, a_{c_g}-\delta)  \le b_g''
    $$
    Thus 
    $$
    g(f(1+\tilde c_g+\delta', a), \tilde c_g+\delta') \le g(b_g'', \tilde c_g+\delta') \le \tilde c_g+\delta'.
    $$
Thus 
$$
 \psi\circ \psi(\cS_{a,b, \tilde c_g} ) \subset \cS_{{\sf i}(\tilde c_g+\delta',a) , {\sf j}(\tilde c_g+\delta',b), \tilde c_g+\delta'}
$$
Thus we are only left to prove the inequalities involving ${\sf i}, {\sf j}$.
    Also,  for the same reason,
    \begin{equation*}
        f(1+\tilde c_g+\delta', a) \le f(1+c_g, a_{c_g} - \delta) 
    \end{equation*}
    and thus
    \begin{equation*}
        {\sf i}(\tilde c_g+\delta', a) \ge  f(1, f(1+ c_g, a_{c_g} - \delta)) = {\sf i}(c_g, a_{c_g} - \delta)  \ge  a_{c_g} - \delta.
    \end{equation*}
    Similarly, 
    \begin{equation*}
        {\sf j}(\tilde c_g+\delta',b) \le {\sf j} (c_g, b_{c_g}+\delta) \le b_{c_g} + \delta. 
    \end{equation*}
     and we are done.
\end{proof}

\subsection{Finding approximations analytically}\label{sec:anal_approx}
We now explain how to analytically find approximations for $c_g,a_{c_g}, b_{c_g}, c_g$ and $\tilde c_g, a_{\tilde c_g}, b_{\tilde c_g}$. Recall the strategy as outlined at the beginning of this section:  all the functions $x \mapsto g(b,x)$, $x \mapsto {\sf i}(c,x), x \mapsto {\sf j}(c,x)$ have unique fixed point in $[0,1)$, and furthermore, the function evaluated at $x$ is at least $x$ if and only if $x $ is smaller than the fixed point. This can be utilized to approximate the parameters in the following sequence of steps.
\begin{itemize}
    \item Find $c_{g,1},c_{g,2}$ so that $g(1,c_{g,1}) >c_{g,1}$ and $g(1,c_{g,2})<c_{g,2}$. 
    \item Next, find $a_{c_{g,1}}$, $a_{c_{g,2}}$ such that ${\sf i}(c_{g,1}, a_{c_{g,1}}) < a_{c_{g,1}}$ and  ${\sf i}(c_{g,2}, a_{c_{g,2}})>a_{c_{g,2}}$. Similarly, find $b_{c_{g,1}}$, $b_{c_{g,2}}$ such that ${\sf j}(c_{g,1}, b_{c_{g,1}}) >b_{c_{g,1}}$ and  ${\sf j}(c_{g,2}, b_{c_{g,2}})<b_{c_{g,2}}$.
\end{itemize}

\begin{lemma}\label{lem:anal_1}
    We have 
    \begin{align*}
        c_{g,1}&\le c_g \le c_{g,2}\\
        a_{c_{g,2}} &\le  a_{c_g}  \le a_{c_{g,1}}\\
        b_{c_{g,1}}& \le  b_{c_g}\le b_{c_{g,2}}.
    \end{align*}
\end{lemma}
\begin{proof}
The first item is immediate from \Cref{lem:fixed_point_of_upper_bound}.    For the inequalities invovling $a$, let $\tilde a_{c_{g,1}}$ denote the unique  fixed point of $x \mapsto {\sf i}(c_{g,1}, x)$ in $(0,1)$. By choice and properties of ${\sf i}$, $a_{c_{g,1}} > \tilde a_{c_{g,1}} $. But
\begin{equation*}
    \tilde a_{c_{g,1}} = {\sf i}(c_{g,1},\tilde a_{c_{g,1}}) = f(1,f(1+c_{g,1}, \tilde a_{c_{g,1}})) \ge {\sf i}(c_g, \tilde a_{c_{g,1}})
\end{equation*}
and hence $a_{c_{g,1}} \ge \tilde a_{c_{g,1}}  > a_{c_g}$. The other inequalities follow exactly the same line of logic, which we skip.
\end{proof}
Now define 
$$
b_{g,1}' = \max(f(1+c_{g,1}, a_{c_{g,1}}), b_{c_{g,1}}) \qquad  b_{g,2}' = \max(f(1+c_{g,2}, a_{c_{g,2}}), b_{c_{g,2}}), 
$$
and then we judiciously choose $\tilde c_{g,1}$ such that $g(b'_{g,1},\tilde c_{g,1}) > \tilde c_{g,1}$ (by judicious, we mean as usual as close as the fixed point as possible) and  similarly choose $\tilde c_{g,2}$ to be such that $g(b'_{g,1},\tilde c_{g,2}) < \tilde c_{g,2}$. Then  define $a_{\tilde c_{g,i}},  b_{\tilde c_{g,i}}$ for $i \in \{1,2\}$ be such that 
$$
{\sf i}(\tilde c_{g,1}, a_{\tilde c_{g,1}}) <  a_{\tilde c_{g,1}} \qquad {\sf i}(\tilde c_{g,2}, a_{\tilde c_{g,2}})>  a_{\tilde c_{g,2}}, \qquad {\sf j}(\tilde c_{g,1}, b_{\tilde c_{g,1}}) >  b_{\tilde c_{g,1}}, \qquad {\sf j}(\tilde c_{g,2}, b_{\tilde c_{g,2}})<  b_{\tilde c_{g,2}}.
$$
The motivation for these definitions is the same  as before: choose an approximation of the fixed point so that the inequaliities are in the good direction. Then
\begin{lemma}\label{lem:anal_2}
    \begin{align*}
    \tilde c_{g,1} &\le  \tilde c_g \le \tilde c_{g,2}\\
     a_{\tilde c_{g,2}} &\le a_{\tilde c_g} \le   a_{\tilde c_{g,1}}\\
        b_{\tilde c_{g,1}} &\le b_{\tilde c_g} \le   b_{\tilde c_{g,2}}
    \end{align*}
\end{lemma}
\begin{proof}
Note $b_{g,1}' \le b_g'$ by choice and \Cref{lem:anal_1}. Let $c_{g,1}'$ be the fixed point of $x \mapsto g(b_{g,1}', x)$. Thus $$\tilde c_{g,1} <c'_{g,1} = g(b_{g,1}', c'_{g, 1}) \le g(b_{g}', c'_{g, 1})$$
which implies $c'_{g,1} \le \tilde c_g$ which is the fixed point of $x \mapsto g(b_g', x)$. Thus $\tilde c_{g,1} \le c_{g,1}' \le \tilde c_g$ as desired. The other direction follows exactly similarly by noting $b_{g,2}' \ge b_g'$.

The inequalities involving $a$ and $b$ can be proven in a similar way to those in \Cref{lem:anal_1} using the bounds for $c$, we leave the details to the reader.
\end{proof}
In reality, we would need a good upper bound for $c_g, \tilde c_g , b_{c_g},\tilde b_{c_g}$ and a good lower bound for $a_{c_g}, a_{\tilde c_g}$. We record these bounds in the following tables.

\begin{table}
    \centering
    \begin{tabular}{|c|c|c|c|c|c|c|}
    \hline
  $d \to$  &2 &3&4&5&6&7  \\\hline

   $a_{c_{g},2}$  &.48&.4 & .3495&  .308&.2734 &.2332 \\ \hline
    $b_{c_g,2}$ &.754 &  .54& .43& .362& .318& .3006\\ \hline
     $c_{g,2}$ &.466 &  .17 & .074 & .0342&.0165&.0081\\ \hline
\end{tabular}
\caption{Approximations of $c_g, a_{c_g}, b_{c_g}$}\label{table:abc_bound}
\end{table}

\begin{table}
    \centering
    \begin{tabular}{|c|c|c|}
    \hline
  $d \to$   &2&7  \\
  \hline
  $b'_{g,2}$ &.755 &.3007\\ \hline
   $a_{\tilde c_{g,2}}$ & .51& .26435\\\hline
    $b_{\tilde c_{g,2}}$  & .664& .26475 \\\hline
     $\tilde c_{g,2}$  &.267 &$3.6\times 10^{-5}$\\\hline
\end{tabular}
\caption{Approximations of $\tilde c_g, a_{\tilde c_{g,2}}, b_{\tilde c_{g_2}}$ $d=2,7$.}\label{table:abc_2nd_iter}
\end{table}
Since we only do even iterations in \Cref{table:abc_bound,table:abc_2nd_iter}, we do one more iteration of $\varphi$ to the values in \Cref{table:abc_bound} for $3 \le d \le 6$ and to that in \Cref{table:abc_2nd_iter} for $d = 2,7$. 
\begin{table}
    \centering
       \begin{tabular}{|c|c|c|c|c|c|c|}
       \hline
  $d \to$  &2 &3&4&5&6&7  \\
  \hline
   $a^o_{c_{g},2}$  &.52&.41 & .3493& .307 &.273 &.26435 \\ \hline
    $b^o_{c_g,2}$ &.665 & .55 &.43 & .362&.318 & .26476\\ \hline
     $c^o_{g,2}$ &.233   &  .063 & .02 &.0016 &.0003&$1.7616\times 10^{-5}$\\\hline 
\end{tabular} 

\caption{Approximations for odd round of iterations}\label{table:abc_odd}
\end{table}

Taking the minimum over rows involving estimates of $a_*$ and maximum over the rows involving the estimates of $b_*, c_*$ from \Cref{table:abc_odd,table:abc_bound} for $3 \le d \le 6$ and \Cref{table:abc_odd,table:abc_2nd_iter} for $d=2,7$ 
we obtain the final table, \Cref{table:abc_final}.

\begin{table}
    \centering
    \begin{tabular}{|c|c|c|c|c|c|c|}
    \hline
  $d \to$  &2 &3&4&5&6&7  \\
  \hline
   $\hat a_*$  &.51&.4 & .3493& .307 &.273 &.26435 \\ \hline
    $\hat b_*$ &.665 & .55 &.43 & .362&.318 & .26476\\ \hline
     $\hat c_*$ &.267   &  .17 & .074 &.0342 &.0165&$3.6\times 10^{-5}$\\ \hline
\end{tabular}
\caption{Final approximations}\label{table:abc_final}
\end{table}



\begin{proof}[Proof of \Cref{prop:estimate_fixed_point}] The reader can check the values in \Cref{table:abc_final} for $\hat a_*$ is bigger than that in \Cref{table:abc_final_section2} and that of $\hat b_*, \hat c_*$ are smaller than that in \Cref{table:abc_final_section2}.
\end{proof}

\subsection{Proof of \Cref{lem:two_fold_iteration}}
In this section we write $f_a(x) = f(a,x)$

The derivative is 
\begin{align}
    f_a'(x)& = -\frac{(2-a)d}{(2x+1)^2}(f(x))^{\frac{d-1}{d}}=-(2-a)d \frac{(1+ax)^{d-1}}{(1+2x)^{d+1}}<0 \text{ for $1 \le a <2$}\label{eq:derivative}\\
    f_a''(x)&=\frac{df(x)(2-a)(a + (2 - a)d + 4ax + 2)}{(2x + 1)^2(ax + 1)^2}>0 \text{ for $1 \le a <2$} \label{eq:double_derivative}
\end{align}
This for this range of $a$, $f$ is decreasing.
\begin{proof}[Proof of \Cref{lem:two_fold_iteration}]
    Note that with the new notation used in this section
    $$
    {\sf i}(x) = f_{1} \circ f_{1+c} (x); \qquad \qquad {\sf j}(x) = f_{1+c} \circ f_1(x)
    $$
One can easily plot the derivatives of these functions for a range of choices of $c$ and check the results for $0<x<1$, and perhaps that is the most practical way to be convinced of the lemma. We could not find a straightforward algebraic way to prove this lemma, so we employ the following rather undesirable case by case analysis.

First observe that by \eqref{eq:derivative}, $f_1$ and $f_{1+c}$ are both decreasing and hence both ${\sf i}$ and ${\sf j} $ are increasing in $[0,1]$.
    We will show that for all $2 \le d \le 7$, $|{\sf i}'(x)|<1$ for all $x \in (0,1)$. This ensures that both ${\sf i}, {\sf j}$ are contractions in $(0,1)$, thus the iterations of ${\sf i}$ and ${\sf j}$ must converge to a fixed point. Furthemore by Rolle's theorem, this fixed point must be unique.

    Our strategy is as follows. Since both $f_1, |f'_1|$ are decreasing in $[0,1]$, and since
    $$
    |f'_{1+c}(x)| \le d\frac{(1+(1+c)x)^{d-1}}{(1+2x)^{d+1}} \le d\frac{(1+(1+c_d)x)^{d-1}}{(1+2x)^{d+1}},
    $$
    for any interval $[a_1,a_2] \subset [0,1]$, we conclude 
    \begin{align*}
           0 \le  {\sf i}'(x) \le  |f_1'(f_{1+c} (a_2))||f'_{1+c}(a_1)| \le |f_1'(f_{1} (a_2))|d\frac{(1+(1+c_d)a_1)^{d-1}}{(1+2a_1)^{d+1}}\text{ for }a_1\le x\le a_2\\
0 \le {\sf j}'(x) \le  |f_{1+c}'(f_{1} (a_2))||f'_{1}(a_1)| \le |f'_{1}(a_1)| d\frac{(1+(1+c_d)f_1(a_2))^{d-1}}{(1+2f_1(a_2))^{d+1}}\text{ for }a_1\le x\le a_2
    \end{align*}
    Denote
   $$
\xi(x) = \xi_{c_d}(x) = d\frac{(1+(1+c_d)x)^{d-1}}{(1+2x))^{d+1}}, \qquad \eta(x) = d\frac{(1+x)^{d-1}}{(1+2x)^{d+1}}.
$$
With this new notation, we have
     \begin{align*}
           0 \le  {\sf i}'(x) &\le \eta(f_1(a_2))\xi(a_1)\text{ for }a_1\le x\le a_2\\
0 \le {\sf j}'(x) &\le   \xi(f_1(a_2))\eta(a_1)\text{ for }a_1\le x\le a_2
    \end{align*}

Clearly $\xi(x) \le \eta(x)$ for all $x \in [0,1]$, thus we can further approximate
\begin{equation}
    \max\{{\sf i}'(x), {\sf j}'(x)\} \le \xi(f_1(a_2))\xi(a_1)\text{ for }a_1\le x\le a_2\label{eq:xi_approx}
\end{equation}
    It is now enough to pick a $c_d>c_g$ and then find a partition $[\alpha_1,\beta_1) \cup [\alpha_2, \beta_2) \cup \ldots \cup [\alpha_k, \beta_k]$ of $[0,1]$ so that $\xi(f_1(\beta_i))\xi(\alpha_i) <1$ for all $1 \le i \le k$.
This can be done by hand. We provide the partition below for $2 \le d \le 7$ and leave it to the diligent reader to check \eqref{eq:xi_approx} holds for each partition. It is interesting to note that the partitions get finer as we move closer to the critical value of $d$ (which is equal to $7$ in this case).
\begin{itemize}
\item For $d=2$, we choose $c_d = .47$. The crudest choice partition $a_1 = 0, a_2=1$ suffices.

\item For $d=3$, we take $c_d = .18$ and we need the partition $$[0,.15] \cup (.15,.65] \cup (.65,1].$$ 

\item For $d=4$, we take $c_d =.08$ and we partition $$[0,.08)\cup[.08,.2) \cup [.2,.41) \cup [.41,1].$$

\item For $d=5$ we take $c_d = .04$ and we partition $$[0,.05)\cup[.05,.1) \cup [.1,.16) \cup [.16,.23] \cup (.23,.33] \cup (.33,.5] \cup (.5,1].$$

\item For $d=6$ we take $c_d = .02$ and we partition \begin{align*}[0,.04)\cup[.04,.08) \cup [.08,.11) \cup [.11,.13]  \cup (.13,.16] \cup (.16,.19]  \\ \cup  (.19, .23] \cup (.23,.27] \cup (.27,.5] \cup [.5,.9) \cup [.9,1].\end{align*}

\item For $d=7$ we take $c_d = .009$ and we partition
\begin{multline*}
    (0,.03] \cup [.03,.07) \cup [.07,.1) \cup [.1,.12) \cup [.12,14) \cup [.14,.15) \cup (\cup_{i=0}^6[.15 + .01i, .16+.01i]) \\\cup [.22,.225) \cup [.225,.23) \cup [.23 , .238) \cup [.238, .245) \cup [.245,.253) \\\cup [.253,.262) \cup [.262, .272) \cup [.272,.28) \cup [.28,.29) \cup [.29,.3) \cup [.3, .31)\\ \cup [.31, .325) \cup [.325,.34) \cup [.34, .365) \cup [.365,.4) \cup [.4,.45) \cup [.45, .55) \cup [.55, .85) \cup [.85,1),
\end{multline*}
\end{itemize}
which completes the proof.
\end{proof}

\bibliographystyle{abbrv}
\bibliography{library}
\end{document}